\documentclass[ejs]{imsart}

\usepackage{amsthm,amsmath,amssymb}
\usepackage{graphicx,type1cm,eso-pic,color}
\RequirePackage[numbers]{natbib}
\RequirePackage[colorlinks,citecolor=blue,urlcolor=blue]{hyperref}

\startlocaldefs
\newtheorem{theorem}{Theorem}[section]
\newtheorem{lemma}[theorem]{Lemma}
\newtheorem{proposition}[theorem]{Proposition}
\newtheorem{corollary}[theorem]{Corollary}

\newtheorem{assumption}[theorem]{Assumption}
\newtheorem{remark}[theorem]{Remark}
\endlocaldefs

\begin{document}
\begin{frontmatter}
\title{On a Nadaraya-Watson Estimator with Two Bandwidths}
\runtitle{On a Nadaraya-Watson Estimator with Two Bandwidths}
\runauthor{F. Comte and N. Marie}
\begin{aug}
\author{\fnms{Fabienne} \snm{COMTE$^{\dag}$}
\ead[label=e1]{fabienne.comte@parisdescartes.fr}}
\and
\author{\fnms{Nicolas} \snm{MARIE$^{\diamond}$}
\ead[label=e2]{nmarie@parisnanterre.fr}}
\address{$^{\dag}$Laboratoire MAP5, Universit\'e de Paris, Paris, France\\
\printead{e1}}
\address{$^{\diamond}$Laboratoire MODAL'X, Universit\'e Paris Nanterre, Nanterre, France\\
\printead{e2}}
\end{aug}
\begin{abstract} In a regression model, we write the Nadaraya-Watson estimator of the regression function as the quotient of two kernel estimators, and propose a bandwidth selection method for both the numerator and the denominator. We prove risk bounds for both data driven estimators and for the resulting ratio. The simulation study confirms that both estimators have good performances, compared to the ones obtained by cross-validation selection of the bandwidth. However, unexpectedly, the single-bandwidth cross-validation estimator is found to be much better than the ratio of the previous two good estimators, in the small noise context. However, the two methods have similar performances in models with large noise.
\end{abstract}
\tableofcontents
\end{frontmatter}
%


%
\section{Introduction}
Consider $n\in\mathbb N\backslash\{0\}$ independent random variables $X_1,\dots,X_n$ having the same probability density $f$ with respect to Lebesgue's measure. Consider also the random variables $Y_1,\dots,Y_n$ defined by
\begin{displaymath}
Y_i = b(X_i) +\varepsilon_i
\textrm{ $;$ }i\in\{1,\dots,n\},
\end{displaymath}
where $b$ is a measurable function from $\mathbb R$ into itself and $\varepsilon_1,\dots,\varepsilon_n$ are $n$ i.i.d. centered random variables with variance $\sigma^2 > 0$ and  independent of $X_1,\dots,X_n$.
\\
Since Nadaraya \cite{NADARAYA64} and Watson \cite{WATSON64}, a lot of consideration has been given to the estimator of $b$ defined by
\begin{displaymath}
\widehat b_{n,h}(x) :=
\frac{\sum_{i = 1}^{n}K\left(\frac{X_i - x}h\right)Y_i}{\sum_{i = 1}^{n}K\left(\frac{X_i - x}{h}\right)}
\textrm{ $;$ }x\in\mathbb R,
\end{displaymath}
where $K :\mathbb R\rightarrow\mathbb R$ is a kernel, and $h > 0$ is the bandwidth. This estimator has been dealt with as a weighted estimator, for $K\geqslant 0$:
\begin{displaymath}
\widehat b_{n,h}(x) = \sum_{i=1}^n w_{n,h}^{(i)}(x) Y_i, \quad w_{n,h}^{(i)}(x)= \frac{K\left(\frac{X_i - x}h\right)}{\sum_{i=1}^n K\left(\frac{X_k - x}h\right)},
\end{displaymath}
and is often called "local average regression". It is studied e.g. in Jones and Wand \cite{JW95}, Gy\"orfi {\it et al.} \cite{GKKW02} or defined in Tsybakov \cite{TSYBAKOV09}. Recent papers still propose methods to improve the estimation, see Chang {\it et al.} \cite{CLW17}. Several strategies have been proposed to select the bandwidth in a data driven way. Cross-validation based on leave-one-out principle is one of the most standard methods to perform this choice (see Gy\"orfi {\it et al.} \cite{GKKW02}), even if a lot of refinements have been proposed. Optimal rates depend on the regularity of the function $b$ and have been first established by Stone \cite{STONE82}: roughly speaking, they are of order $O(n^{-p/(2p + 1)})$ for $b$ admitting $p$ derivatives. From theoretical point of view, the rates of the adaptive final estimator are not always given, nor proved.
\\
In this paper, we re-write the Nadaraya-Watson as the quotient of two estimators, an estimator of $bf$ divided by an estimator of $f$:
\begin{displaymath}
\widehat{bf}_{n,h}(x) :=
\frac{1}{nh}\sum_{i = 1}^{n}K\left(\frac{X_i - x}{h}\right)Y_i
\end{displaymath}
and 
\begin{displaymath}
\widehat f_{n,h'}(x) :=
\frac{1}{nh'}\sum_{i = 1}^{n}K\left(\frac{X_i - x}{h'}\right).
\end{displaymath}
Clearly, $\widehat f_{n,h'}$ is the Parzen-Rosenblatt estimator of $f$ (see Rosenblatt \cite{ROSENBLATT56} and Parzen \cite{PARZEN62}). The question we are interested in is the following one: can we choose separately the two bandwidths in an adaptive way and obtain good performance for each, and then for the ratio? This is why we study the estimator 
\begin{displaymath}
\widehat b_{n,h,h'}(x) :=
\frac{\widehat{bf}_{n,h}(x)}{\widehat f_{n,h'}(x)}
\textrm{ $;$ }x\in\mathbb R
\end{displaymath}
as an estimator of the regression function $b$, where $h,h' > 0$ and $K :\mathbb R\rightarrow\mathbb R$ is a (not necessarily nonnegative) kernel. Thus, $\widehat b_{n,h,h}=\widehat b_{n,h}$ is the initial Nadaraya-Watson estimator of $b$ with single bandwidth $h$. For this reason, the estimator studied in this paper is called the \textit{two bandwidths Nadaraya-Watson} (2bNW) estimator.
\\
Adaptive estimation of the density has been widely studied recently. A bandwidth selection method  has been proposed  by Goldenschluger and Lepski \cite{GL11}, and proved to reach the adequate bias-variance compromise. Implementation of this method revealed to be difficult due to the choice of two constants involved in the procedure, the intuition of which is not obvious. This is why the question was further investigated by Lacour {\it et al.} \cite{LMR17}: they improve and modify the strategy by using specific theoretical tools for their proofs. Precisely, thanks to a deviation inequality for U-statistics proved by Houdr\'e and Reynaud-Bouret \cite{HRB03}, they bound the Mean Integrated Square Error of their final estimator, which they call PCO (Penalised Comparison to Overfitting) estimator. Numerically, the good performance of their proposal has been illustrated in a naive way and for high order kernels in Comte and Marie \cite{CM20}, and through a systematic numerical study in Lacour~{\it et al.} \cite{LMRV}, including the multivariate case. These two methods and the associated results are dedicated to the selection of $h'$ for $\widehat f_{n,h'}(x)$, and we can use them. Unfortunately, the theoretical results do not apply to $\widehat{bf}_{n,h}(x)$, mainly because they hold under a boundedness assumption: in our context, this would lead to assume that the $Y_i$'s are bounded. We do not want to require such an assumption as it would exclude the case of Gaussian errors $\varepsilon_i$, for instance. Thus, we give moment assumptions under which the Goldenshluger and Lepski method on the one hand (see Section \ref{section_GL}) and the PCO estimator on the other hand (see Section \ref{LMR}) can be applied to the estimation of $bf$. When gathering the results for the numerator and the denominator, we can bound the risk of the quotient estimator of $b$.
\\
Concretely, we implement the PCO method for $bf$ and compare it with a cross-validation (CV) strategy: in our examples, PCO almost always performs slightly better than CV. Therefore, the PCO adaptive estimation strategies for $f$ and for $bf$ are clearly good. However, unexpectedly, for small noise ($\sigma=0.1$), the quotient fails systematically to beat the specific regression CV method. Even if we compare the classical single-bandwidth CV regression estimator to the ratio of the oracles estimators of the numerator and the denominator, the former wins, and we obtain a quotient with two bandwidths which is in mean much less good than the CV estimator with single bandwidth. In practice, the  bandwidth  selected by the CV algorithm in that case is very small, and associated to quite bad estimators of the numerator and of the denominator. This remark is of important interest for practitioners. In a second time, we increased the noise ($\sigma=0.7$), and finally obtained results indicating that the two methods can have similar Mean Integrated Squared Errors (MISE) in this more difficult context. This, together with the fact we establish a theoretical risk bound on the PCO adaptive 2bNW estimator, imply  that the PCO method, for both numerator and denominator, remains an interesting bandwidth selection method. Moreover, we believe that both positive but also negative results are of interest, and detailed tables, explanations and discussion are given in Section \ref{section_simu}.
\\
\\
\textbf{Notations:}
\begin{enumerate}
 \item For every square integrable functions $f,g :\mathbb R\rightarrow\mathbb R$,
 \begin{displaymath}
 (f\ast g)(x) :=
 \int_{-\infty}^{\infty}
 f(x - y)g(y)dy
 \textrm{ $;$ }
 x\in\mathbb R.
 \end{displaymath}
 \item $K_{\varepsilon} := 1/\varepsilon K(\cdot /\varepsilon)$ for every $\varepsilon > 0$.
\end{enumerate}
%


%
\section{Bound on the MISE of the 2bNW estimator}\label{section_MISE}
First, we state some simple risk bound results in the case of a fixed bandwidth.
\\
\\
Consider $\beta > 0$ and $\ell :=\lfloor\beta\rfloor$, where $\lfloor\beta\rfloor$ denotes the largest integer smaller than $\beta$. In the sequel, the kernel $K$ and the density function $f$ fulfill the following assumption.
%


%
\begin{assumption}\label{assumption_K_f_numerator}$\;$\\
\vspace{-0.5cm}
\begin{itemize}
 \item[(i)] The map $K$ belongs to $\mathbb L^2(\mathbb R,dy)$, $K$ is bounded and $\int_{{\mathbb R}}K(y)dy = 1$.
 \item[(ii)] The density function $f$ is bounded.
\end{itemize}
\end{assumption}
\noindent
Under this assumption, a suitable control of the MISE of $\widehat{bf}_{n,h}$ has been established in Comte \cite{COMTE17}, Proposition 4.2.1.
%


%
\begin{proposition}\label{bound_MISE_numerator_2bNW}
Under Assumption \ref{assumption_K_f_numerator},
\begin{displaymath}
\mathbb E(\|\widehat{bf}_{n,h} - bf\|_{2}^{2})
\leqslant
\|bf - (bf)_h\|_{2}^{2} +\frac{\mathfrak c_{K,Y}}{nh}
\end{displaymath}
where $(bf)_h := K_h\ast(bf)$ and $\mathfrak c_{K,Y} :=\|K\|_{2}^{2}\mathbb E(Y_{1}^{2})$.
\end{proposition}
\noindent
In order to provide a suitable control on the MISE of the 2bNW estimator, we assume that $b$ and $f$ fulfill the following assumption.
%


%
\begin{assumption}\label{assumption_b_f}
The function $b^2f$ is bounded by a constant $\mathfrak c_{b,f} > 0$.
\end{assumption}
\noindent
Note that this assumption does not require that $b$ is bounded and is satisfied in most classical examples.
\\
Moreover, for any $\mathcal S\in\mathcal B(\mathbb R)$, consider the norm $\|.\|_{2,f,\mathcal S}$ on $\mathbb L^2(\mathcal S, f(x)dx)$ defined by
\begin{displaymath}
\|\varphi\|_{2,f,\mathcal S} :=
\left(\int_{\mathcal S}\varphi(x)^2f(x)dx\right)^{1/2}
\textrm{$;$ }
\forall\varphi\in\mathbb L^2(\mathcal S,f(x)dx).
\end{displaymath}
%


%
\begin{proposition}\label{bound_MISE_2bNW}
Let $m_n$ be a positive real number and consider
\begin{displaymath}
\mathcal S_n :=\{x\in\mathbb R : f(x)\geqslant m_n\}.
\end{displaymath}
Under Assumptions \ref{assumption_K_f_numerator} and \ref{assumption_b_f},
\begin{displaymath}
\mathbb E(\|\widehat b_{n,h,h'} - b\|_{2,f,\mathcal S_n}^{2})
\leqslant
\frac{8\mathfrak c_f}{m_{n}^{2}}\left(
\|bf - (bf)_h\|_{2}^{2} +\frac{\mathfrak c_{K,Y}}{nh}
+ 2\mathfrak c_{b,f}\left(
\|f - f_{h'}\|_{2}^{2} +\frac{\mathfrak c_K}{nh'}\right)\right)
\end{displaymath}
where $(bf)_h := K_h\ast(bf)$, $f_{h'} = K_{h'}\ast f$, $\mathfrak c_f :=\|f\|_{\infty}^{2}\vee 1$ and $\mathfrak c_K := \int_{{\mathbb R}}K(y)^2dy.$
\end{proposition}
\noindent
The idea behind Proposition \ref{bound_MISE_2bNW} is that we cannot pretend to accurately estimate $b$ on domains where few $X_i$'s are observed. Such domains correspond to small level of the density. For small $m_n$, the set ${\mathcal S}_n$ excludes these cases.
\\
Proposition \ref{bound_MISE_2bNW} gives a decomposition of the risk of the quotient estimator as the sum of the risks of the estimators of the numerator $bf$ and of the denominator $f$, up to the multiplicative constant $8\mathfrak c_f/m_{n}^{2}$. Therefore, the rate of the quotient estimator is, in the best case, the worst rate of the two estimators used to define it (see also Remark \ref{Nikol} below). The factor $1/m_{n}^{2}$ may imply a global loss with respect to this rate. Clearly, the smaller is $m_n$, the larger is the loss.
\\
For instance, if $f$ is lower bounded by a known constant $f_0$ on a given compact set $A$, then we can take ${\mathcal S}_n = A$ and $m_n=f_0$. In that case, no loss occurs. If $f_0$ is unknown, we still can bound the risk with ${\mathcal S}_n=A$ and $1/m_n^2= \log(n)$ for $n$ large enough. A log-loss occurs then in the rate.
%


%
\begin{remark}\label{Nikol} We consider, for  $\beta, L > 0$,  the Nikol'ski ball ${\mathcal H}(\beta, L)$, defined as the set of $\ell =\lfloor\beta\rfloor$ times continuously derivable functions $\varphi :\mathbb R\rightarrow\mathbb R$ such that $\varphi^{(\ell)}$ satisfies
\begin{displaymath}
\left[\int_{-\infty}^{\infty}
(\varphi^{(\ell)}(x + t) -\varphi^{(\ell)}(x))^2dx\right]^{1/2}
\leqslant L|t|^{\beta -\ell}
\textrm{ $;$ }
\forall t\in\mathbb R.
\end{displaymath}
For instance, for $p\in\mathbb N$, any function $\varphi\in C^{p + 1}(\mathbb R)$ such that ${\rm supp}(\varphi^{(p)}) = [0,1]$ and $\|\varphi^{(p + 1)}\|_{\infty}\leqslant L$ belongs to $\mathcal H(p + 1,L)$. Indeed, for every $t\in\mathbb R_+$,
\begin{eqnarray*}
 \int_{-\infty}^{\infty}
 (\varphi^{(p)}(x + t) -\varphi^{(p)}(x))^2dx
 & \leqslant &
 t\int_{-t}^{1}\int_{x}^{x + t}\varphi^{(p + 1)}(z)^2\mathbf 1_{[0,1]}(z)dzdx\\
 & \leqslant &
 L^2t\int_{-t}^{1}((x + t)\wedge 1 - x\vee 0)dx\\
 & = &
 L^2t\left(\int_{0}^{1 + t}(x\wedge 1)dx -\int_{0}^{1}xdx\right) = L^2t^2.
\end{eqnarray*}
More subtly, $\psi : x\mapsto e^{-x}\mathbf 1_{\mathbb R_+}(x)$ belongs to $\mathcal H(1/2,1)$. Indeed, for every $t\in\mathbb R_+$,
\begin{eqnarray*}
 \int_{-\infty}^{\infty}
 (\psi(x + t) -\psi(x))^2dx
 & = &
 \int_{-t}^{\infty}e^{-2(x + t)}dx -
 2\int_{0}^{\infty}e^{-t - 2x}dx +\int_{0}^{\infty}e^{-2x}dx\\
 & = &
 \lim_{x\rightarrow\infty}
 -\frac{1}{2}[e^{-2t}(e^{-2x} - e^{2t})
 - 2e^{-t}(e^{-2x} - 1)\\
 & &
 \hspace{3cm}
 + e^{-2x} - 1] = 1 - e^{-t}\leqslant t.
\end{eqnarray*}
Now, assume that $bf$ belongs to $\mathcal H(\beta_1,L)$ and $f$ to $\mathcal H(\beta_2,L)$. We also assume that the kernel $K$ satisfies Assumption \ref{assumption_K_f_numerator} and is of order $\ell =\lfloor\max(\beta_1,\beta_2)\rfloor$, that is
\begin{displaymath}
\int_{-\infty}^{\infty}
|u^kK(u)|du <\infty
\quad\textrm{and}\quad
\int_{-\infty}^{\infty}
u^kK(u)du = 0
\textrm{ $;$ }\forall k\in\{1,\dots,\ell\}.
\end{displaymath}
Then, it follows from Tsybakov \cite{TSYBAKOV09}, Chapter 1, that
\begin{displaymath}
\|bf - (bf)_h\|_{2}^{2}\leqslant C(\beta_1, L)h^{2\beta_1}
\quad\textrm{and}\quad
\|f - f_{h'}\|_{2}^{2}\leqslant C'(\beta_2, L)(h')^{2\beta_2}.
\end{displaymath}
This implies that choosing $h_{{\rm opt}} = c_1n^{1/(2\beta_1 + 1)}$ in Proposition \ref{bound_MISE_numerator_2bNW} yields
\begin{displaymath}
\mathbb E(\|\widehat{bf}_{n,h_{\rm opt}} - bf\|_{2}^{2})
\lesssim n^{-2\beta_1/(2\beta_1 + 1)},
\end{displaymath}
which is a standard optimal rate of estimation on Nikol'ski balls. The same rate holds for the estimation of $f$ under our assumptions, with $\beta_1$ replaced by $\beta_2$, and $h_{\rm opt}' = c_2n^{1/(2\beta_2 + 1)}$. This implies that 
\begin{eqnarray*}
 & &
 \|bf - (bf)_{h_{\rm opt}}\|_{2}^{2} +\frac{\mathfrak c_{K,Y}}{nh_{\rm opt}}
 + 2\mathfrak c_{b,f}\left(
 \|f - f_{h_{\rm opt}'}\|_{2}^{2} +\frac{\mathfrak c_K}{nh_{\rm opt}'}\right)\\
 & &
 \hspace{5cm}\lesssim 
 \max(n^{-2\beta_1/(2\beta_1 + 1)},n^{-2\beta_2/(2\beta_2 + 1)}).
\end{eqnarray*}
So, the rate is optimal if $\beta =\min(\beta_1,\beta_2)$ is the regularity of $b$.\\
However, such bandwidth choices are not possible in practice, as they depend on unknown regularity parameters. Data driven bandwidth selection methods are settled to automatically reach a squared bias-variance compromise, inducing the optimal rate if the function under estimation does belong to a regularity space.
\end{remark}
%


%
\section{A bandwidth selection procedure for the 2bNW estimator based on the GL method}\label{section_GL}
The bound on the MISE of $\widehat b_{n,h,h'}$ obtained in Proposition \ref{bound_MISE_2bNW} suggests to select $h$ and $h'$ separately, so that both bounds are minimal. The Goldenshluger-Lepski method (see Goldenshluger and Lepski \cite{GL11}) allows to do this for $\widehat{f}_{n,h'}$, but requires to be extended to the estimator of $bf$. In particular, extensions of the proof are required as we do not wish to assume that the $Y_i$'s are bounded. 
\\
\\
Consider the collection of bandwidths $\mathcal H_n :=\{h_1,\dots,h_{N(n)}\}\subset [0,1]$, where $N(n)\in\{1,\dots,n\}$ and
\begin{displaymath}
\frac{1}{n} < h_1 <\dots < h_{N(n)}.
\end{displaymath}
Moreover, we will need the following conditions.
%


%
\begin{assumption}\label{bandwidth_GL}
There exists $\mathfrak m > 0$, not depending on $n$, such that
\begin{displaymath}
\frac{1}{n}\sum_{i = 1}^{N(n)}\frac{1}{h_i}\leqslant\mathfrak m,
\end{displaymath}
and for every $c > 0$, there exists $\mathfrak m(c) > 0$, not depending on $n$, such that
\begin{displaymath}
\sum_{i = 1}^{N(n)}
\frac{1}{\sqrt{h_i}}
\exp\left(-\frac{c}{\sqrt{h_i}}\right)
\leqslant\mathfrak m(c).
\end{displaymath}
\end{assumption}
\noindent
{\bf Example.} Consider the dyadic bandwidths defined by
\begin{displaymath}
h_i := 2^{-i}
\textrm{ $;$ }
\forall i = 0,1,\dots,\left[\frac{\log(n)}{\log(2)}\right].
\end{displaymath}
Then,
\begin{displaymath}
\frac{1}{n}\sum_{i = 1}^{[\log(n)/\log(2)]}2^i
\leqslant\frac{2n - 1}n\leqslant 2
\end{displaymath}
and 
\begin{displaymath}
\sum_{i = 1}^{[\log(n)/\log(2)]}2^{i/2}\exp(-c2^{i/2})\leqslant
\sum_{i = 1}^{n}\sqrt i\exp(-c\sqrt i)\leqslant\mathfrak m(c) <\infty.
\end{displaymath}
Thus, Assumption \ref{bandwidth_GL} is fulfilled.
\\
\\
Consider also
\begin{eqnarray*}
 \widehat{bf}_{n,h,\eta}(x) & := &
 (K_{\eta}\ast\widehat{bf}_{n,h})(x)\\
 & = &
 \frac{1}{n}\sum_{i = 1}^{n}Y_i(K_{\eta}\ast K_h)(X_i - x).
\end{eqnarray*}
We apply the Goldenshluger-Lepski bandwidth selection method to $\widehat{bf}_{n,h}$ by solving the minimization problem
\begin{equation}\label{definition_GL}
\min_{h\in\mathcal H_n}\{A_n(h) + V_n(h)\}
\end{equation}
where
\begin{displaymath}
A_n(h) :=\sup_{\eta\in\mathcal H_n}
(\|\widehat{bf}_{n,h,\eta} -
\widehat{bf}_{n,\eta}\|_{2}^{2} - V_n(\eta))_+
\quad 
\textrm{ and }\quad 
V_n(h) :=
\upsilon
\frac{\mathfrak c_{K,Y}}{nh}\|K\|_{1}^{2},
\end{displaymath}
with $\upsilon > 0$ not depending on $n$ and $h$, and $\mathfrak c_{K,Y} =\|K\|_{2}^{2}\mathbb E(Y_{1}^{2})$. In the sequel, the solution to the minimization Problem (\ref{definition_GL}) is denoted by $\widehat h_n$.\\
The idea behind the criterion is that $A_n(h)$ is an estimate of the squared bias term $\|(bf)_h - bf\|_{2}^{2}$ and $V_n(h)$ an estimate of the variance. So, $\widehat h_n$ makes the compromise. See more details about the heuristics in Chagny \cite{CHAGNY16}, Section 4.4.
%


%
\begin{theorem}\label{bound_GL}
Under Assumptions \ref{assumption_K_f_numerator} and \ref{bandwidth_GL}, if $\mathbb E(Y_{1}^{6}) <\infty$, then there exist two deterministic constants $\mathfrak c,\overline{\mathfrak c} > 0$, not depending on $n$, such that
\begin{displaymath}
\mathbb E(\|\widehat{bf}_{n,\widehat h_n} - bf\|_{2}^{2})
\leqslant
\mathfrak c\cdot\inf_{h\in\mathcal H_n}\{\|(bf)_h - bf\|_{2}^{2} + V_n(h)\} +
\overline{\mathfrak c}\frac{\log(n)^2}{n}.
\end{displaymath}
\end{theorem}
\noindent
Theorem \ref{bound_GL} states that $\widehat{bf}_{n,\widehat h_n}$ automatically leads to a compromise between the squared bias ($\|(bf)_h - bf\|_{2}^{2}$) and the variance ($V_n(h)$) terms. The multiplicative constant $\mathfrak c$, which is larger than one, is the price of the method but preserves the rate. Lastly, the additive quantity $\overline{\mathfrak c}\log(n)^2/n$ is negligible with respect to the possible rate of convergence (see Remark \ref{Nikol}).
\\
We recall now a version of the result proved by Goldenshluger and Lepski \cite{GL11}, which is available for the estimator of $f$. See also a simplified proof in Comte \cite{COMTE17}, Section 4.2. Let us consider
\begin{displaymath}
\widehat h_n'\in
\arg\min_{h'\in\mathcal H_n}\{A_n'(h') + V_n'(h')\},
\end{displaymath}
where
\begin{displaymath}
A_n'(h') :=\sup_{\eta\in\mathcal H_n}
(\|K_{\eta}\ast\widehat f_{n,h'} -
\widehat f_{n,\eta}\|_{2}^{2} - V_n'(\eta))_+
\quad\textrm{and}\quad
V_n'(h') :=
\chi\frac{\|K\|_{2}^{2}\|K\|_{1}^{2}}{nh'}
\end{displaymath}
with $\chi > 0$ not depending on $n$ and $h'$. Under Assumptions \ref{assumption_K_f_numerator} and \ref{bandwidth_GL}, there exist two deterministic constants $\mathfrak c',\overline{\mathfrak c}' > 0$, not depending on $n$, such that
\begin{equation}\label{GL_density}
\mathbb E(\|\widehat f_{n,\widehat h_n'} - f\|_{2}^{2})
\leqslant
\mathfrak c'\cdot\inf_{h'\in\mathcal H_n}\{\|f_{h'} - f\|_{2}^{2} + V_n'(h')\} +
\frac{\overline{\mathfrak c}'}{n}
\end{equation}
\noindent
Gathering (\ref{GL_density}) and Theorem \ref{bound_GL} yields a Corollary similar to Proposition \ref{bound_MISE_2bNW}.
%


%
\begin{corollary}\label{bound_2bNW_GL}
Let $m_n$ be a positive real number and consider
\begin{displaymath}
\mathcal S_n :=\{x\in\mathbb R : f(x)\geqslant m_n\}.
\end{displaymath}
Under Assumptions \ref{assumption_K_f_numerator}, \ref{assumption_b_f} and \ref{bandwidth_GL}, if $\mathbb E(Y_{1}^{6}) <\infty$, then
\begin{eqnarray*}
 \mathbb E(\|\widehat b_{n,\widehat h_n,\widehat h_n'} - b\|_{2,f,\mathcal S_n}^{2})
 & \leqslant &
 \mathfrak C_n
 \inf_{(h,h')\in\mathcal H_{n}^{2}}\{\|(bf)_h - bf\|_{2}^{2} +\|f_{h'} - f\|_{2}^{2}\\
 & &
 \hspace{2cm}
 + V_n(h) + V_n'(h')\}
 +\overline{\mathfrak C}_n\frac{\log(n)^2}{n},
\end{eqnarray*}
where
\begin{displaymath}
\mathfrak C_n :=\frac{8\mathfrak c_f}{m_{n}^{2}}(\mathfrak c\vee(2\mathfrak c_{b,f}\mathfrak c'))
\quad\textrm{and}\quad
\overline{\mathfrak C}_n :=
\frac{8\mathfrak c_f}{m_{n}^{2}}(
\overline{\mathfrak c} + 2\mathfrak c_{b,f}\overline{\mathfrak c}').
\end{displaymath}
\end{corollary}
\noindent
The comments following Proposition \ref{bound_MISE_2bNW} and in Remark \ref{Nikol} apply here.
%


%
\section{A bandwidths selection procedure for the 2bNW estimator based on the PCO method}\label{LMR}
The Goldenshluger-Lepski method is mathematically very nice and provides a rigorous risk bound for the adaptive estimator with random bandwidth. However, it has been acknowledged as being difficult to implement, due to the square grid in $h,\eta$ required to compute intermediate versions of the criterion and to the lack of intuition to guide the choice of the constants $\upsilon$ and $\chi$ which should be calibrated from preliminary simulation experiments, see e.g. Comte and Rebafka \cite{CR16}. This is the reason why Lacour {\it et al.} \cite{LMR17} investigated and proposed a simplified criterion (PCO) relying on deviation inequalities for $U$-statistics due to Houdr\'e and Reynaud-Bouret \cite{HRB03}. This inequality applies in our more complicated context and Lacour-Massart-Rivoirard's result can be extended here as follows.
\\
\\
Let us recall that $K_h(\cdot) = 1/hK(\cdot/h)$ and
\begin{displaymath}
(bf)_h =
\mathbb E(\widehat{bf}_{n,h}) =
K_h\ast (bf)
\end{displaymath}
(see Lemma \ref{preliminaries_GL}). Let $h_{\min}$ be the smallest bandwidth value in $\mathcal H_n$ and consider
\begin{displaymath}
\textrm{crit}(h) :=
\|\widehat{bf}_{n,h} -\widehat{bf}_{n,h_{\min}}\|_{2}^{2} +\textrm{pen}(h)
\end{displaymath}
with
\begin{displaymath}
\textrm{pen}(h) :=
\frac{2\langle K_{h_{\min}},K_h\rangle_2}{n^2}
\sum_{i = 1}^{n}Y_{i}^{2}.
\end{displaymath}
Then, let us define
\begin{displaymath}
\widetilde h_n\in\arg\min_{h\in\mathcal H_n}\textrm{crit}(h).
\end{displaymath}
The idea behind the proposal of Lacour {\it et al.}~(2017) is that, instead of comparing estimators $\widehat{bf}_{n,h}$ to a collection of estimators $\widehat{bf}_{n,h,\eta}$ for different bandwidths $\eta$, it is sufficient to compare them to the same single estimator, corresponding to the smallest bandwidth. See their Section 3.1 for more heuristic elements. This implies a faster and more efficient numerical procedure.\\
In the sequel, in addition to Assumption \ref{assumption_K_f_numerator}, the kernel $K$, the functions $b$ and $f$, the distribution of $Y_1$ and $h_{\min}$ fulfill the following assumption.
%


%
\begin{assumption}\label{assumption_K_LMR}
The kernel $K$ is symmetric and $K(0) > 0$,
\begin{displaymath}
\frac{1}{nh_{\min}}\leqslant 1,
\end{displaymath}
$bf$ is bounded, and there exists $\alpha > 0$ such that $\mathbb E(\exp(\alpha|Y_1|)) <\infty$.
\end{assumption}
\noindent
As for Assumption \ref{assumption_b_f}, we can note that assuming $bf$ bounded does not require $b$ to be bounded, since most densities decrease fast at infinity.  Moreover, the moment condition here is $\mathbb E(\exp(\alpha|Y_1|)) <\infty$ and is stronger than for the Goldenschluger and Lepski method ($\mathbb E(Y_{1}^{6}) <\infty$).
%


%
\begin{theorem}\label{bound_LMR}
Consider $\vartheta\in (0,1)$. Under Assumptions \ref{assumption_K_f_numerator} and \ref{assumption_K_LMR}, there exist two deterministic constants $\mathfrak a,\mathfrak b > 0$, not depending on $n$, $h_{\min}$ and $\vartheta$, such that
\begin{displaymath}
\mathbb E(\|\widehat{bf}_{n,\widetilde h_n} - bf\|_{2}^{2})
\leqslant
(1 +\vartheta)\inf_{h\in\mathcal H_n}\mathbb E(\|\widehat{bf}_{n,h} - bf\|_{2}^{2})
+\frac{\mathfrak a}{\vartheta}\|(bf)_{h_{\min}} - bf\|_{2}^{2} +
\frac{\mathfrak b}{\vartheta}\cdot\frac{\log(n)^5}{n}.
\end{displaymath}
\end{theorem}
\noindent
Theorem \ref{bound_LMR} states that the estimator $\widehat{bf}_{n,\widetilde h_n}$ has performance  of order of the best estimator of the collection $\inf_{h\in\mathcal H_n}\mathbb E(\|\widehat{bf}_{n,h} - bf\|_{2}^{2})$ up to a factor $(1 +\vartheta)$. Indeed, the two other terms can be considered as negligible. If $bf$ is in the Nikol'ski ball ${\mathcal H}(\beta_1,L)$ as in Remark \ref{Nikol}, then the first right-hand-side term is of order $n^{-2\beta_1/(2\beta_1 + 1)}$. Since for $h_{\min} = 1/n$, $\|(bf)_{h_{\min}} - bf\|_{2}^{2}$ is of order $n^{-2\beta_1}$, both this term and the last residual term $\log(n)^5/n$ are negligible compared to the first one.
\\
\\
Now, we state the result that can be deduced from Lacour {\it et al.} \cite{LMR17} for the estimator of $f$. Let us consider
\begin{displaymath}
\widetilde h_n'\in
\arg\min_{h'\in\mathcal H_n}\textrm{crit}'(h'),
\end{displaymath}
where
\begin{displaymath}
\textrm{crit}'(h') :=
\|\widehat f_{n,h'} -\widehat f_{n,h_{\min}}\|_{2}^{2} +\textrm{pen}'(h')
\quad\textrm{and}\quad
\textrm{pen}'(h') :=
\frac{2\langle K_{h_{\min}},K_h\rangle_2}{n}.
\end{displaymath}
By Lacour {\it et al.} \cite{LMR17}, Theorem 2, there exists two deterministic constants $\mathfrak a',\mathfrak b' > 0$, not depending on $n$ and $h_{\min}$, such that for every $\vartheta\in (0,1)$,
\begin{eqnarray*}
 \mathbb E(\|\widehat f_{n,\widetilde h_n'} - f\|_{2}^{2})
 & \leqslant &
 (1 +\vartheta)\inf_{h'\in\mathcal H_n}\mathbb E(\|\widehat f_{n,h'} - f\|_{2}^{2}) +
 \frac{\mathfrak a'}{\vartheta}\|f_{h_{\min}} - f\|_{2}^{2} +
 \frac{\mathfrak b'}{\vartheta n}.
\end{eqnarray*}
\noindent
Again, we can gather this last result and Theorem \ref{bound_LMR} to get the following Corollary.
%


%
\begin{corollary}\label{bound_2bNW_LMR}
Let $m_n$ be a positive real number and consider
\begin{displaymath}
\mathcal S_n :=\{x\in\mathbb R : f(x)\geqslant m_n\}.
\end{displaymath}
Consider also $\vartheta\in (0,1)$. Under Assumptions \ref{assumption_K_f_numerator}, \ref{assumption_b_f} and \ref{assumption_K_LMR},
\begin{eqnarray*}
 \mathbb E(\|\widehat{bf}_{n,\widetilde h_n,\widetilde h_n'} - bf\|_{2,f,\mathcal S_n}^{2})
 & \leqslant &
 (1 +\vartheta)\mathfrak C_n(1,1)
 \inf_{(h,h')\in\mathcal H_{n}^{2}}\{\mathbb E(\|\widehat{bf}_{n,h} - bf\|_{2}^{2})\\
 & &
 \hspace{4cm}
 +\mathbb E(\|\widehat f_{n,h'} - f\|_{2}^{2})\}\\
 & &
 +\frac{\mathfrak C_n(\mathfrak a,\mathfrak a')}{\vartheta}(\|(bf)_{h_{\min}} - bf\|_{2}^{2}
 +\|f_{h_{\min}} - f\|_{2}^{2})\\
 & &
 +\frac{\mathfrak C_n(\mathfrak b,\mathfrak b')}{\vartheta}\cdot\frac{\log(n)^5}{n},
\end{eqnarray*}
where
\begin{displaymath}
\mathfrak C_n(u,v) :=\frac{8\mathfrak c_f}{m_{n}^{2}}(u\vee(2\mathfrak c_{b,f}v))
\textrm{ $;$ }
\forall u,v\in\mathbb R.
\end{displaymath}
\end{corollary}
\noindent
The proof of Corollary \ref{bound_2bNW_LMR} relies to the same arguments as the proof of Corollary \ref{bound_2bNW_GL} provided in Section \ref{bound_2bNW_GL}, and is therefore omitted.
%


%
\section{Simulation study}\label{section_simu}
For the noise, we consider $\varepsilon\sim\sigma\mathcal N(0,1)$, with $\sigma = 0.1$ and $\sigma = 0.7$. For the signal, we take either $X\sim\mathcal N(0,1)$ or $X\sim\gamma(3,2)/5$ (where the factor $5$ is set to keep the variance of $X$ of order $1$, as in the first case). For the function $b$, we took functions with different features and regularities: 
\begin{itemize}
 \item $b_1(x) =\exp(-x^2/2)$, 
 \item $b_2(x) = x^2/4-1$,
 \item $b_3(x) =\sin(\pi x)$,
 \item $b_4(x) =\exp(-|x|)$.
\end{itemize}
We illustrate in Figures \ref{fig1} and \ref{fig2} the difference between a sample generated with $\sigma = 0.1$ (small noise) and with $\sigma = 0.7$ (large noise), compared with the functions to estimate. We can see that the first case is easy and that the second one is very difficult. Notice that the vertical scales are different.
\begin{figure}[h!]
\includegraphics[width=12cm,height=8cm]{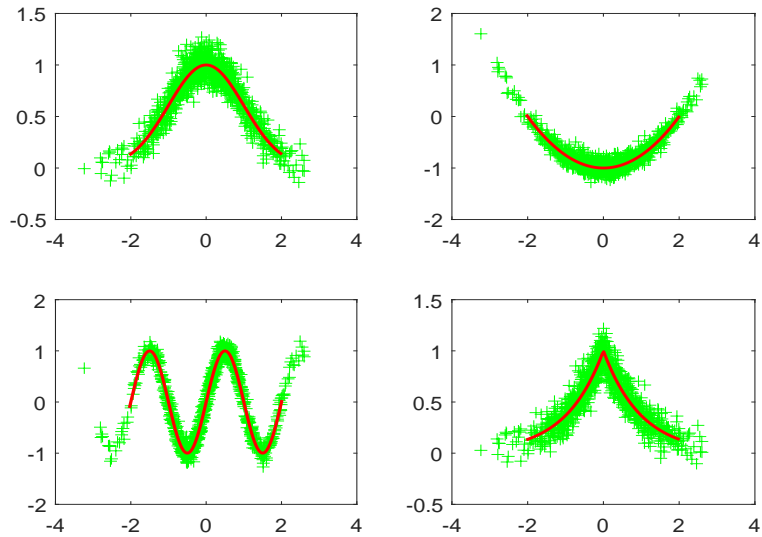}
\caption{Observations $(X_i,Y_i)_{1\leqslant i\leqslant n}$ for $n = 1000$ in the four cases of functions $b_1$ to $b_4$, with small noise $\sigma = 0.1$, and true regression function in bold red.}
\label{fig1}
\end{figure}
\begin{figure}[h!]
\includegraphics[width=12cm,height=8cm]{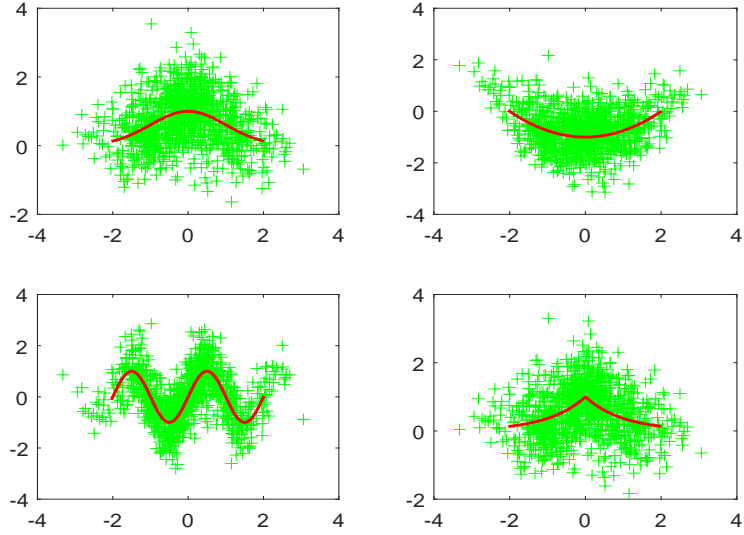}
\caption{Observations $(X_i,Y_i)_{1\leqslant i\leqslant n}$ for $n = 1000$ in the four cases of functions $b_1$ to $b_4$, with large noise $\sigma = 0.7$, and true regression function in bold red.}
\label{fig2}
\end{figure}
%


%
\subsection{Estimation of $bf$}
The PCO method is implemented for $f$ and $bf$ with a kernel of order $7$ (i.e. $\int x^kK(x)dx = 0$ for $k = 1$ to $7$), defined by $K(x) = 4n_1(x) - 6n_2(x) + 4n_3(x) - n_4(x)$, where $n_j(.)$ is a Gaussian density with mean $0$ and variance $j$. Note that, for $n_{j,h}(x) := 1/hn_j(x/h)$, it holds that
\begin{equation}\label{scalar}
\langle n_{i,h_1},n_{j,h_2}\rangle_2 =
\int_{-\infty}^{\infty}n_{i,h_1}(x)n_{j,h_2}(x)dx =
\frac{1}{\sqrt{2\pi}}\times\frac{1}{\sqrt{ih_{1}^{2} + jh_{2}^{2}}}
\end{equation}
The bandwidth is selected among $M = 75$ equispaced values in between $0.01$ and $1$. All functions (true or estimated) are computed at 100 equispaced points in the interquantile interval corresponding to the $2\%$ and $98\%$ quantiles of $X$. The bandwidth is selected via the PCO criterion, where $h_{\rm min} = 0.01$, and
\begin{displaymath}
\textrm{crit}(h) :=
\|\widehat{bf}_{n,h} -\widehat{bf}_{n,h_{\min}}\|_{2}^{2} +\textrm{pen}(h)
\quad\textrm{with}\quad
\textrm{pen}(h) :=
\frac{2\langle K_{h_{\min}},K_h\rangle_2}{n^2}
\sum_{i = 1}^{n}Y_{i}^{2}.
\end{displaymath}
Note that the bandwidth of the density estimator is selected as in Comte and Marie \cite{CM20}, by minimizing
\begin{displaymath}
\textrm{crit}'(h) :=
\|\widehat{f}_{n,h} -\widehat{f}_{n,h_{\min}}\|_{2}^{2} + 2\textrm{pen}'(h)
\quad\textrm{with}\quad
\textrm{pen}'(h) :=
\frac{2\langle K_{h_{\min}},K_h\rangle_2}{n}.
\end{displaymath}
The $\mathbb L^2$-norm is computed as a Riemann sum on the interquantile interval, while the penalty is explicit and exact, thanks to Formula (\ref{scalar}).\\
The cross-validation (CV) criterion for selecting the bandwidth of $\widehat{bf}_{n,h}$ is computed as follows:
\begin{displaymath}
CV(h) :=
\int\widehat{bf}_h(x)^2dx -\frac{2}{n(n - 1)h}
\sum_{i = 1}^{n}
\sum_{j = 1,j\neq i}^{n}Y_iY_jN\left(\frac{X_i - X_j}h\right),
\end{displaymath}
where $N(.)$ is the Gaussian kernel, also used to compute the estimator $\widehat{bf}_{n,h}$ in this case. It provides an estimation of $\|\widehat{bf}_h\|_{2}^{2} - 2\langle\widehat{bf}_h,bf\rangle_2$ relying on the idea that the empirical for $\langle t,bf\rangle_2$ is $1/n\sum_{i = 1}^{n}Y_it(X_i)$. The chosen bandwidth is the minimizer of $CV(h)$ in the same collection as previously.
\begin{table}[h!]
\hspace{-0cm}\begin{tabular}{c|ccc|ccc}
 & \multicolumn{3}{c}{$b_1f$} & \multicolumn{3}{c}{$b_2f$}\\
 $n$ & PCO & CV & Or & PCO & CV & Or\\ \hline
 $250$ & 0.33 & 0.37 & 0.16 & 0.32 & 0.38 & 0.15\\
 & {\small (0.28)} & {\small (0.82)} & {\small (0.14)} & {\small (0.28)} & {\small (1.01)}   & {\small (0.14)}\\
 $500$ & 0.17 & 0.29 & 0.08 & 0.17 & 0.26 & 0.07\\
 & {\small (0.13)} & {\small (0.86)} & {\small (0.07)} & {\small (0.14)} & {\small (1.10)} & {\small (0.07)}\\
 $1000$ & 0.09 & 0.21 & 0.05 & 0.09 & 0.21 & 0.04\\
 & {\small (0.07)} & {\small (0.57)} & {\small (0.04)} & {\small (0.07)} & {\small (0.60)} & {\small (0.03)}\\ \multicolumn{6}{c}{}\\
\end{tabular}
\begin{tabular}{c|ccc|ccc}
 & \multicolumn{3}{c}{$b_3f$} & \multicolumn{3}{c}{$b_4f$}\\
 $n$ & PCO & CV & Or & PCO & CV & Or\\ \hline
 $250$ & 0.45 & 0.42 & 0.31 & 0.35 & 0.40 & 0.17\\
 & {\small (0.25)} & {\small (0.32)} & {\small (0.17)} & {\small (0.24)} & {\small (0.81)}  & {\small (0.17)}\\
 $500$ & 0.23 & 0.35 & 0.15 & 0.20 & 0.39 & 0.11\\
 & {\small (0.14)} & {\small (0.88)} & {\small (0.08)} & {\small (0.13)} & {\small (0.84)} & {\small (0.06)}\\
 $1000$ & 0.12 & 0.24 & 0.09 & 0.11 & 0.19 & 0.07\\
 & {\small (0.07)} & {\small (0.55)} & {\small (0.05)} & {\small (0.07)} & {\small (0.40)} & {\small (0.04)}\\ \multicolumn{6}{c}{}\\
\end{tabular}
\caption{100*MISE (with 100*std in parenthesis below) for the estimation of $bf$ corresponding to the four examples $b_1,\dots, b_4$, 200 repetitions, $X\sim {\mathcal N}(0,1)$ and $\sigma = 0.1$. Columns PCO and CV correspond to the two competing methods. "Or" is for "oracle" and gives the average error of the best possible estimator of the collection, computed for each sample.}\label{tab_bf_XGauss}
\end{table}
\begin{table}[h!]
\hspace{-0cm}\begin{tabular}{c|ccc|ccc}
 & \multicolumn{3}{c}{$b_1f$} & \multicolumn{3}{c}{$b_2f$}\\
 $n$ & PCO & CV & Or & PCO & CV & Or\\ \hline
 $250$ & 0.56 & 0.70 & 0.30 & 0.51 & 0.92 & 0.28\\
 & {\small (0.40)} & {\small (1.44)} & {\small (0.22)} & {\small (0.33)} & {\small (2.83)} & {\small (0.22)}\\
 $500$ & 0.30 & 0.32 & 0.16 & 0.29 & 0.33 & 0.16\\
 & {\small (0.23)} & {\small (0.68)} & {\small (0.13)} & {\small (0.23)} & {\small (0.60)} & {\small (0.13)}\\
 $1000$ & 0.14 & 0.29 & 0.08 & 0.15 & 0.30 & 0.08\\
 & {\small (0.10)} & {\small (0.91)} & {\small (0.06)} & {\small (0.15)} & {\small (1.14)} & {\small (0.07)}\\ \multicolumn{6}{c}{}\\
\end{tabular}
\begin{tabular}{c|ccc|ccc}
 & \multicolumn{3}{c}{$b_3f$} & \multicolumn{3}{c}{$b_4f$}\\
 $n$ & PCO & CV & Or & PCO & CV & Or\\ \hline
 $250$ & 0.91 & 0.85 & 0.61 & 0.67 & 0.71 & 0.39\\
 & {\small (0.60)} & {\small (0.72)} & {\small (0.37)} & {\small (0.37)} & {\small (1.09)} & {\small (0.21)}\\
 $500$ & 0.45 & 0.47 & 0.32 & 0.38 & 0.33 & 0.22\\
 & {\small (0.27)} & {\small (0.77)} & {\small (0.19)} & {\small (0.21)} & {\small (0.21)} & {\small (0.13)}\\
 $1000$ & 0.21 & 0.30 & 0.16 & 0.21 & 0.27 & 0.12\\
 & {\small (0.11)} & {\small (0.63)} & {\small (0.08)} & {\small (0.12)} & {\small (0.58)} & {\small (0.06)}\\ \multicolumn{6}{c}{}\\
\end{tabular}
\caption{100*MISE (with 100*std in parenthesis below) for the estimation of $bf$ corresponding to the four examples $b_1,\dots, b_4$, 200 repetitions, $X\sim {\mathcal N}(0,1)$ and $\sigma = 0.7$. Columns PCO and CV correspond to the two competing methods. "Or" is for "oracle" and gives the average error of the best possible estimator of the collection, computed for each sample.}\label{tab_bf_XGaussSig07}
\end{table}
\begin{table}
\hspace{-0cm}\begin{tabular}{c|ccc|ccc}
 & \multicolumn{3}{c}{$b_1f$} & \multicolumn{3}{c}{$b_2f$}\\
 $n$ & PCO & CV & Or & PCO & CV & Or\\ \hline
 $250$ & 0.71 & 0.54 & 0.54 & 0.78 & 0.62 & 0.62\\   
 & {\small (0.14)} & {\small (0.15)} & {\small (0.08)} & {\small (0.15)} & {\small (0.17)} & {\small (0.09)}\\
 $500$ & 0.62 & 0.47 & 0.51 & 0.70 & 0.54 & 0.28\\
 & {\small (0.13)} & {\small (0.14)} & {\small (0.08)} & 
{\small (0.15)} & {\small (0.18)} & {\small (0.08)}\\
 $1000$ & 0.55 & 0.41 & 0.47 & 0.63 & 0.50 & 0.54\\
 & {\small (0.11)} & {\small (0.14)} & {\small (0.07)} & {\small (0.12)} & {\small (0.17)} &  {\small (0.07)}\\
\multicolumn{6}{c}{}\\
\end{tabular}
\begin{tabular}{c|ccc|ccc}
 & \multicolumn{3}{c}{$b_3f$} & \multicolumn{3}{c}{$b_4f$}\\
 $n$ & PCO & CV & Or & PCO & CV & Or\\ \hline
 $250$ & 0.35 & 0.28 & 0.31 & 0.57 & 0.39 & 0.37\\   
 & {\small (0.04)} &  {\small (0.05)} & {\small (0.03)} & {\small (0.17)} & {\small (0.14)} & {\small (0.07)}\\
 $500$ & 0.32 & 0.26 & 0.28 & 0.48 & 0.32 & 0.33\\
 & {\small (0.04)} & {\small (0.06)} & {\small (0.03)} & {\small (0.13)} & {\small (0.14)} & {\small (0.07)}\\
 $1000$ & 0.30 & 0.24 & 0.26 & 0.39 & 0.28 & 0.28\\
 & {\small (0.03)} & {\small (0.07)} & {\small (0.03)} & {\small (0.11)} & {\small (0.11)} & {\small (0.06)}\\
\multicolumn{6}{c}{}\\
\end{tabular}
\caption{Means of selected bandwidths (with std in parenthesis below) for the estimation of $bf$, 200 repetitions, $X\sim {\mathcal N}(0,1)$, $\sigma = 0.1$.}\label{tab_bf_bandwidth_XGauss}
\end{table}
\noindent
Tables \ref{tab_bf_XGauss} and \ref{tab_bf_XGaussSig07} give the MISE obtained for 200 repetitions and sample sizes $250$, $500$ and $1000$, for the estimation of $bf$ with PCO and CV methods, for $\sigma = 0.1$ (Table \ref{tab_bf_XGauss}) and $\sigma = 0.7$ (Table \ref{tab_bf_XGaussSig07}). The column "Or" gives the mean of the minimal squared errors for each sample, which requires to use the unknown true function and represents what could be obtained at best (that is if the best possible bandwidth was chosen for each sample). We postpone results with $X\sim\gamma(3,2)/5$ in Appendix \ref{Append} since they are similar. We can see that the PCO method is globally better than the CV, with no important difference, and the oracle shows that we are in the right orders even if not at best. 
\\
Table \ref{tab_bf_bandwidth_XGauss} presents the mean of the selected bandwidths in each case PCO and CV, and allows to compare it with the oracle bandwidth, for the same paths and configurations as previously. The conclusion here is that, in mean, the PCO method over-estimates the oracle bandwidth, while the CV method slightly under-evaluates it. Clearly, the too-large choice gives better results.
%


%
\subsection{Estimation of $b$}
Now, we present the results for the estimation of the regression function $b$, obtained either with a single-bandwidth estimator, or with the ratio of two adaptive PCO estimators of $bf$ and $f$. 
\\
The PCO estimators are the ones studied above, which proved to be good estimators (see also  the study for the estimation of $f$ in Comte and Marie \cite{CM20}). We simply take a point by point ratio of the two adaptive PCO estimators. The oracle we refer to is  computed with the estimator of $b$ obtained as a quotient of the two oracles of $bf$ and $f$ for each path. It is the best performance we can expect with a PCO-ratio strategy.
\\
For the one-bandwidth Nadaraya-Watson estimator $\widehat{b}_{n,h}$, it is computed with the Gaussian kernel $N(.)$. The leave-one-out cross-validation criterion which is minimized for the bandwidth selection is
\begin{displaymath}
CV_{\rm NW}(h) :=
\sum_{i = 1}^{n}(Y_i - b_{n,h}^{(-i)}(X_i))^2
\end{displaymath}
with
\begin{displaymath}
b_{n,h}^{(-i)}(x) :=
\sum_{j = 1,j\neq i}^{n}\frac{N((X_j - x)/h)}{\sum_{k = 1,k\neq i}^{n}N((X_k - x)/h)}Y_j.
\end{displaymath}
%


%
\subsubsection{Small noise case}
We've started the study with $\sigma = 0.1$, which in our mind was an easy case (see Figure \ref{fig1}).
\begin{table}
\hspace{-0cm}\begin{tabular}{c|ccc|ccc}
 & \multicolumn{3}{c}{$b_1$} & \multicolumn{3}{c}{$b_2$}\\
 $n$ & CV & PCO & Or & CV & PCO & Or\\ \hline
 $250$ & 0.34 & 2.80 & 1.15 & 0.43 & 3.41 & 1.37\\
 & {\small (0.19)} & {\small (2.80)} & {\small (1.04)} & {\small (0.24)} & {\small (3.98)} &  {\small (2.77)}\\ 
 $500$ & 0.19 & 1.44 & 0.61 & 0.23 & 1.49 & 0.58\\
 & {\small (0.08)} & {\small (1.05)} & {\small (0.53)} & {\small (0.12)} & {\small (1.83)} & {\small (1.28)}\\
 $1000$ & 0.10 & 0.74 & 0.37 & 0.13 & 0.53 & 0.26\\
 & {\small (0.05)} & {\small (0.51)} & {\small (0.31)} & {\small (0.05)}  & {\small (0.58)} & {\small (0.28)}\\
 \multicolumn{6}{c}{}\\
\end{tabular}
\begin{tabular}{c|ccc|ccc}
 & \multicolumn{3}{c}{$b_3$} & \multicolumn{3}{c}{$b_4$}\\
 $n$ & CV & PCO & Or & CV & PCO & Or\\ \hline
 $250$ & 1.34 & 7.93 & 6.30 & 0.39 & 2.87 & 1.22\\
 & {\small (0.75)} & {\small (5.09)} & {\small (4.72)} & {\small (0.18)} & {\small (2.07)} & {\small (0.69)}\\ 
 $500$ & 0.66 & 4.42 & 2.96 & 0.22 & 1.72 & 0.81\\
 & {\small (0.31)} & {\small (2.58)} & {\small (2.09)} & {\small (0.09)} & {\small (0.94)} & {\small (0.45)}\\
 $1000$ & 0.30 & 2.30 & 1.65 & 0.12 & 0.92 & 0.48\\
 & {\small (0.08)} & {\small (1.24)} & {\small (1.01)} & {\small (0.04)} & {\small (0.55)} & {\small (0.21)}\\
 \multicolumn{6}{c}{}\\
\end{tabular}
\caption{100*MISE (with 100*std in parenthesis below) for the estimation of $b_i$, $i = 1, \dots,4$, 200 repetitions, $X\sim {\mathcal N}(0,1)$, $\sigma = 0.1$. CV and PCO are the two competing methods. Column "Or" gives the average of ISE for the ratio of the two best estimators of $bf$ and $f$ in the collection.}\label{tab_b_XGauss}
\end{table}
\begin{table}
\begin{tabular}{c|cccc}
 $n$ & $b_1$ & $b_2$ & $b_3$ & $b_4$\\ \hline
 $250$ & 0.13 & 0.13 & 0.06 & 0.10\\
 & {\small (0.02)} & {\small (0.03)} & {\small (0.01)} & {\small (0.02)}\\
 $500$ & 0.12 & 0.11 & 0.05 & 0.09\\
 & {\small (0.02)} & {\small (0.02)} & {\small (0.01)} & {\small (0.01)}\\
 $1000$ & 0.11 & 0.09 & 0.05 & 0.08\\
 & {\small (0.01)} & {\small (0.02)} & {\small (0.01)} & {\small (0.01)}\\ \multicolumn{5}{c}{}\\
\end{tabular}
\caption{Mean of selected bandwidth (with std in parenthesis below) with the CV method for NW-single bandwidth estimator of $b$, $\sigma = 0.1$.}\label{tab_b_bandwithNW}
\end{table}
\begin{figure}
\includegraphics[width=12cm,height=5cm]{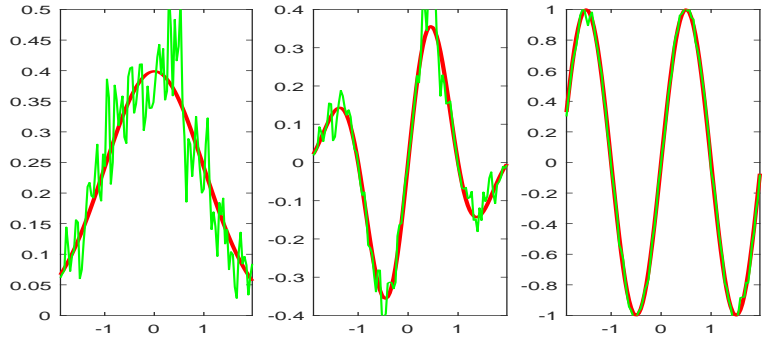}
\caption{Small bandwidth effect, the ratio of two bad estimators is a good estimator. For bandwidth $h = 0.01$ and $n = 1000$, true (bold) and estimated $f$ (left), $b_1f$ (middle), and $b_1$ (right).}\label{fig3}
\end{figure}
\noindent
Table \ref{tab_b_XGauss} presents the results for the estimation of $b$, either with the $CV_{{\rm NW}}$ criterion, with ratio of PCO of $bf $ and $f$, or with the ratio of the best estimators of $bf$ and $f$ in the collection. More precisely, the column "Or" gives here the MISE computed with the estimator of $b$ obtained as a quotient of the two oracles of $bf$ and $f$ in each example and for each sample path. Clearly, the performance of the Nadaraya-Watson cross-validation criterion is much better, within a multiplicative factor from 2 and up to 6. The variance of the quotient estimators (oracle and PCO) are large, which shows that the mean performance is probably deteriorated by a few very bad results. However, the result is puzzling: even the ratio of the two best estimators of the numerator and denominator does not reach the good performance of the single-bandwidth CV method. Table \ref{tab_b_bandwithNW} shows in addition that the selected bandwidths are in mean very small. We can check that the ratio of this bad numerator divided by a bad denominator fits well to the $b$ quotient function: this is illustrated by Figure \ref{fig3}. It is likely that both imply a compensation resulting in a locally, and thus also globally, better estimate. We can notice that the selected bandwidth also decrease more slowly when $n$ increases (see Table \ref{tab_b_bandwithNW}) than for the estimator of $bf$ (see Table \ref{tab_bf_bandwidth_XGauss}). Our explanation (see the heuristic Remark \ref{heuristic_weights_variance} below) is that the risk of the Nadaraya-Watson estimator behaves as $C(h^{2\alpha} +  \sigma^2/(nh))$, for some $\alpha>0$ related to the regularity of $b$, like in the projection least-squares method (see e.g. Baraud \cite{BARAUD02}). In the small noise case, $\sigma^2$ makes the variance term negligible, so that the bandwidth selection method aims at having small bias term $h^{2\alpha}$. On the other hand, the risk decomposition of the estimator of $bf$ involves a variance term of order $\|K\|_{2}^{2}\mathbb E(Y_{1}^{2})/(nh)$, and in all our examples, empirical evaluations of $\mathbb E(Y^2)$ is in the range $[0.34,0.70]$, making the ratio with $\sigma^2$ between 34 and 70. In other words, the variance term for this estimator is 34 to 70 times larger. This is why it is important to investigate large noise case and a less favorable  signal to noise ratio.
%


%
\subsubsection{Large noise case}
When setting $\sigma=0.7$, the empirical order of ${\mathbb E}(Y^2)$ for the four models is between 0.91 and 1.31, which divided by $\sigma^2$ gives now a value between 1.85 and 2.67. This is much smaller than previously. This corresponds to a more difficult estimation problem, as can be seen from Figure \ref{fig2}.
\begin{table}[!h]
\hspace{-0cm}\begin{tabular}{c|ccc|ccc}
 & \multicolumn{3}{c}{$b_1$} & \multicolumn{3}{c}{$b_2$}\\
 $n$ & CV & PCO & Or & CV & PCO & Or\\ \hline
 $250$ & 7.81 & 9.02 & 6.54 & 8.32 & 8.61 & 5.18\\
 & {\small (5.32)} & {\small (5.86)} & {\small (5.13)} & {\small (6.02)} & {\small (6.67)} &  {\small (4.87)}\\
$500$ & 4.33 & 4.34  & 3.25 & 4.57 & 4.83 & 3.02\\
 & {\small (3.09)} & {\small (2.71)} & {\small (2.15)} & {\small (2.86)} & {\small (4.07)} & {\small (2.58)}\\
 $1000$ & 2.26 & 2.19 & 1.72 & 2.38 & 2.13 & 1.40\\
 & {\small (1.37)} & {\small (1.30)} & {\small (1.16)} & {\small (1.39)} & {\small (2.04)} & {\small (1.15)}\\ \multicolumn{6}{c}{}\\
\end{tabular}
\begin{tabular}{c|ccc|ccc}
 & \multicolumn{3}{c}{$b_3$} & \multicolumn{3}{c}{$b_4$}\\
 $n$ & CV & PCO & Or & CV & PCO & Or\\ \hline
 $250$ & 16.2 & 18.5 & 14.8 & 7.84 & 9.36 & 7.89\\
 & {\small (8.47)} & {\small (10.3)} & {\small (9.06)} & {\small (4.94)} & {\small (4.62)} & {\small (4.94)}\\ 
$500$ & 8.54 & 9.28 & 7.58 & 4.41 & 5.01 & 4.23\\
 & {\small (3.58)} & {\small (5.26)} & {\small (4.41)} & {\small (2.83)} & {\small (2.31)} & {\small (2.27)}\\
 $1000$ & 4.54 & 4.62 & 3.87 & 2.40 & 2.85 & 2.35\\
 & {\small (1.81)} & {\small (2.35)} & {\small (2.18)} & {\small (1.28)} & {\small (1.42)} & {\small (1.20)}\\
 \multicolumn{6}{c}{}\\
\end{tabular}
\caption{100*MISE (with 100*std in parenthesis below) for the estimation of $b_i$, $i = 1, \dots, 4$, 200 repetitions, $X\sim {\mathcal N}(0,1)$, $\sigma = 0.7$. CV and PCO are the two competing methods. Column "Or" gives the average of ISE for the ratio of the two best estimators of $bf$ and $f$ in the collection.}\label{tab_b_XGaussSig07}
\end{table}
\newline
We now comment the results given in Table \ref{tab_b_XGaussSig07}. The MISE are quite larger, but in Figure \ref{fig4}, we show examples of estimated curves in this case, and the associated orders of MISEs, computed for 25 repetitions; they are not as good as for small noise, but still reasonable. The results in Table \ref{tab_b_XGaussSig07} show that the MISE have now the same orders, and the oracles can be much better than the results of the Nadarya-Watson estimator. The selected bandwidths are larger and decreasing with $n$ (see Table \ref{tab_b_bandwithNWsig07} in Appendix).
\\
\\
The conclusion of this study is that adaptive estimation of functions with kernel estimators and bandwidth selection relying on the PCO method proposed by Lacour {\it et al.} \cite{LMR17} gives very good results in theory and practice, not only for density estimation. However, for regression function estimation, one bandwidth selected with a criterion directly suited to the regression function is safer than the two different bandwidths selected when considering the Nadaraya-Watson estimator as a quotient of two functions that may be estimated separately.  The results are not bad, but the strategy must be devoted to more complicated contexts where direct estimators of $b$ are not feasible.
%


%
\begin{remark}\label{heuristic_weights_variance}
Note that if $h = h'$, which means that $\widehat b_{n,h,h'}$ is the usual Nadaraya-Watson estimator $\widehat b_{n,h}$,
\begin{displaymath}
\widehat b_{n,h}(x) =
\sum_{i = 1}^{n}w_{n,h}^{(i)}(x)\varepsilon_i
\quad
{\rm with}\quad
w_{n,h}^{(i)}(x) :=\frac{K((X_i - x)/h)}{\sum_{j = 1}^{n}K((X_j - x)/h)}.
\end{displaymath}
Then,
\begin{displaymath}
\widehat b_{n,h}(x) - b(x) =
\sum_{i = 1}^{n}w_{n,h}^{(i)}(x)(b(X_i) - b(x)) +
\sum_{i = 1}^{n}w_{n,h}^{(i)}(x)\varepsilon_i,
\end{displaymath}
and for a nonnegative kernel with compact support $[-1,1]$, if the regression function $b$ is Lispchitz continuous, then
\begin{displaymath}
\mathbb E[(\widehat b_{n,h}(x) - b(x))^2]
\leqslant
Ch^2 +
\sigma^2\mathbb E\left(\sum_{i = 1}^{n}w_{n,h}^{(i)}(x)^2\right).
\end{displaymath}
Moreover,
\begin{displaymath}
\mathbb E\left(\sum_{i = 1}^{n}w_{n,h}^{(i)}(x)^2\right) =
\frac{1}{nh}\mathbb E\left(
\frac{\frac{1}{nh}\sum_{i = 1}^{n}K((X_i - x)/h)^2}{[\frac{1}{nh}\sum_{i = 1}^{n}K((X_i - x)/h)]^2}\right),
\end{displaymath}
and for a fixed $h > 0$, by the law of large numbers,
\begin{displaymath}
\frac{1}{nh}\sum_{i = 1}^{n}K\left(\frac{X_i - x}{h}\right)^2
\xrightarrow[n\rightarrow\infty]{{\rm a.s.}}
\frac{1}{h}\mathbb E\left[K\left(\frac{X_1 - x}{h}\right)^2\right] =
\int_{-\infty}^{\infty}K(u)^2f(x+uh)du
\end{displaymath}
and
\begin{eqnarray*}
 \left[\frac{1}{nh}\sum_{i = 1}^{n}K\left(\frac{X_i - x}{h}\right)\right]^2
 & \xrightarrow[n\rightarrow\infty]{{\rm a.s.}} &
 \mathbb E\left[\frac{1}{h}
 K\left(\frac{X_1 - x}{h}\right)\right]^2\\
 & &
 \hspace{2cm}
 =\left[\int_{-\infty}^{\infty}K(u)f(x + uh)du\right]^2.
\end{eqnarray*}
Then, for small $h$, the first limit has order $\|K\|_{2}^{2}f(x)$ and the second one has order $f^2(x)$. To sum up, the risk of $\widehat b_{n,h}(x)$ is heuristically of order $Ch^2 +\sigma^2\|K\|_{2}^{2}f(x)/(nh)$. This explains why, for small $\sigma^2$, the variance term gets small and the estimator can choose small bandwidth to make the bias as small as possible.
\end{remark}
\begin{figure}[h!]
\includegraphics[width=12cm,height=8cm]{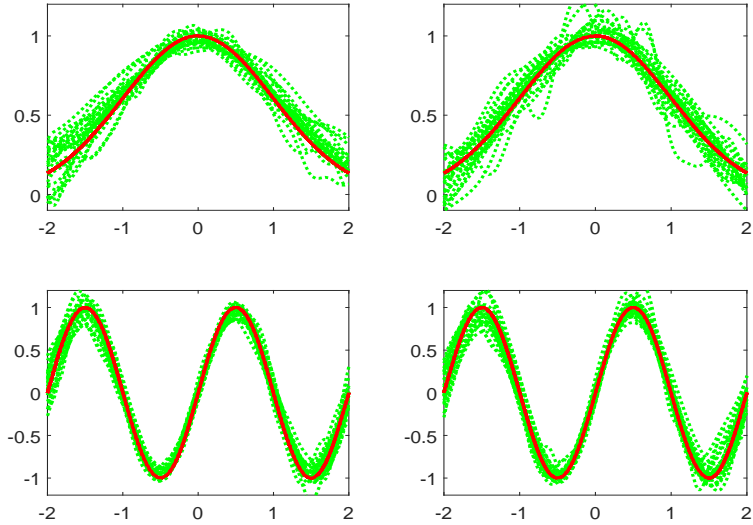}
\caption{Example of 25 estimated $b$ with large noise $\sigma = 0.7$, $n = 1000$, true $b$ in bold red, estimated $b$ with sign-bandwidth CV selection (dotted green, left) and with ratio of PCO (dotted green, right) for functions 1 and 3. 100*MSE$_{(100\;std)}$ are
$2.62_{(1.64)}$ and $2.87_{(2.41)}$ (top) ; $4.71_{(1.74)}$ and $5.49_{(2.85)}$.}\label{fig4}
\end{figure}
%


%
\section{Proofs}\label{section_proofs}
%


%
\subsection{Proof of Proposition \ref{bound_MISE_2bNW}}
On the one hand, by Comte \cite{COMTE17}, Proposition 3.3.1,
\begin{equation}\label{bound_MISE_2bNW_1}
\mathbb E(\|\widehat f_{n,h'} - f\|_{2}^{2})
\leqslant
\|f - f_{h'}\|_{2}^{2} +\frac{\mathfrak c_K}{nh'}
\end{equation}
and, by Proposition \ref{bound_MISE_numerator_2bNW},
\begin{equation}\label{bound_MISE_2bNW_2}
\mathbb E(\|\widehat{bf}_{n,h} - bf\|_{2}^{2})
\leqslant
\|bf - (bf)_h\|_{2}^{2} +\frac{\mathfrak c_{K,Y}}{nh}
\end{equation}
For the proof of Inequality (\ref{bound_MISE_2bNW_1}), the reader can also refer to Tsybakov \cite{TSYBAKOV09}. On the other hand,
\begin{displaymath}
\widehat b_{n,h,h'} - b =
\left(\frac{\widehat{bf}_{n,h} - bf}{\widehat f_{n,h'}} +
bf\left(\frac{1}{\widehat f_{n,h'}} -\frac{1}{f}\right)\right)\mathbf 1_{\widehat f_{n,h'}(.) > m_n/2} -
b\mathbf 1_{\widehat f_{n,h'}(.)\leqslant m_n/2}.
\end{displaymath}
Then,
\begin{eqnarray*}
 & &
 \|\widehat b_{n,h,h'} - b\|_{2,f,\mathcal S_n}^{2}\\
 & &
 \hspace{1cm}\leqslant
 \frac{8\mathfrak c_1}{m_{n}^{2}}\left(\|\widehat{bf}_{n,h} - bf\|_{2}^{2}
 +\int_{-\infty}^{\infty}b(x)^2f(x)|\widehat f_{n,h'}(x) - f(x)|^2dx
 \right)\\
 & &
 \hspace{2cm}
 + 2\int_{\mathcal S_n}
 b(x)^2f(x)\mathbf 1_{|\widehat f_{n,h'}(x) - f(x)| > m_n/2}dx
\end{eqnarray*}
with $\mathfrak c_1 :=\|f\|_{\infty}\vee\|f\|_{\infty}^{2}$.
\\
\\
By Markov's inequality,
\begin{eqnarray*}
 \mathbb E(\|\widehat b_{n,h,h'} - b\|_{2,f,\mathcal S_n}^{2})
 & \leqslant &
 \frac{8\mathfrak c_1}{m_{n}^{2}}(\mathbb E(\|\widehat{bf}_{n,h} - bf\|_{2}^{2}) +
 \mathfrak c_{b,f}
 \mathbb E(\|\widehat f_{n,h'} - f\|_{2}^{2}))\\
 & &
 + 2\mathfrak c_{b,f}\int_{\mathcal S_n}
 \mathbb P\left(|\widehat f_{n,h'}(x) - f(x)| >\frac{m_n}{2}\right)dx\\
 & \leqslant &
 \frac{8(\mathfrak c_1\vee 1)}{m_{n}^{2}}(\mathbb E(\|\widehat{bf}_{n,h} - bf\|_{2}^{2}) +
 2\mathfrak c_{b,f}\mathbb E(\|\widehat f_{n,h'} - f\|_{2}^{2})).
\end{eqnarray*}
Inequalities (\ref{bound_MISE_2bNW_1}) and (\ref{bound_MISE_2bNW_2}) allow to conclude.
%


%
\subsection{Proof of Theorem \ref{bound_GL}}
First, let us prove the following lemma.
%


%
\begin{lemma}\label{preliminaries_GL}
Consider
\begin{displaymath}
(bf)_h :=
K_h\ast (bf)
\quad\textrm{and}\quad
(bf)_{h,\eta} :=
K_{\eta}\ast K_h\ast (bf).
\end{displaymath}
Then,
\begin{displaymath}
\mathbb E(\widehat{bf}_{n,h}(x)) =
(bf)_h(x)
\quad\textrm{and}\quad
\mathbb E(\widehat{bf}_{n,h,\eta}(x)) =
(bf)_{h,\eta}(x).
\end{displaymath}
\end{lemma}
%


%
\begin{proof}
Since $\mathbb E(\varepsilon_k) = 0$ and $X_k$ and $\varepsilon_k$ are independent for every $k\in\{1,\dots,n\}$,
\begin{eqnarray*}
 \mathbb E(\widehat{bf}_{n,h}(x)) & = &
 \frac{1}{n}
 \sum_{k = 1}^{n}\mathbb E(b(X_k)K_h(X_k - x)) +
 \frac{1}{n}
 \sum_{k = 1}^{n}\mathbb E(\varepsilon_k)\mathbb E(K_h(X_k - x))\\
 & = &
 \int_{-\infty}^{\infty}K_h(y - x)b(y)f(y)dy =
 (bf)_h(x)
\end{eqnarray*}
and
\begin{eqnarray*}
 \mathbb E(\widehat{bf}_{n,h,\eta}(x)) & = &
 \frac{1}{n}
 \sum_{k = 1}^{n}\mathbb E(b(X_k)(K_{\eta}\ast K_h)(X_k - x))\\
 & &
 +\frac{1}{n}
 \sum_{k = 1}^{n}\mathbb E(\varepsilon_k)\mathbb E((K_{\eta}\ast K_h)(X_k - x))\\
 & = &
 \int_{-\infty}^{\infty}(K_{\eta}\ast K_h)(y - x)b(y)f(y)dy =
 (bf)_{h,\eta}(x).
\end{eqnarray*}
\end{proof}
\noindent
Since
\begin{displaymath}
\widehat h_n\in\arg\min_{h\in\mathcal H_n}\{A_n(h) + V_n(h)\},
\end{displaymath}
for every $h\in\mathcal H_n$,
\begin{equation}\label{bound_GL_1}
\mathbb E(\|\widehat{bf}_{n,\widehat h_n} - bf\|_{2}^{2})
\leqslant
3\mathbb E(\|\widehat{bf}_{n,h} - bf\|_{2}^{2}) + 6V_n(h) + 6\mathbb E(A_n(h))
\end{equation}
Let us find a suitable control of $\mathbb E(A_n(h))$. First of all, for any $h,\eta\in\mathcal H_n$,
\begin{eqnarray*}
 \|\widehat{bf}_{n,h,\eta} -\widehat{bf}_{n,\eta}\|_{2}^{2}
 & \leqslant &
 3(\|\widehat{bf}_{n,h,\eta} - (bf)_{h,\eta}\|_{2}^{2}\\
 & &
 +\|\widehat{bf}_{n,\eta} - (bf)_{\eta}\|_{2}^{2}
 +\|(bf)_{h,\eta} - (bf)_{\eta}\|_{2}^{2}).
\end{eqnarray*}
Then,
\begin{eqnarray}
 \label{bound_GL_2}
 A_n(h) & \leqslant &
 3\left[\sup_{\eta\in\mathcal H_n}
 \left(\|\widehat{bf}_{n,h,\eta} -
 (bf)_{h,\eta}\|_{2}^{2} -\frac{V_n(\eta)}{6}\right)_+\right.\\
 & &
 \left.
 +\sup_{\eta\in\mathcal H_n}
 \left(\|\widehat{bf}_{n,\eta} -
 (bf)_{\eta}\|_{2}^{2} -\frac{V_n(\eta)}{6}\right)_+
 +\|(bf)_{h,\eta} - (bf)_{\eta}\|_{2}^{2}\right].
 \nonumber
\end{eqnarray}
On the one hand,
\begin{displaymath}
\|(bf)_{h,\eta} - (bf)_{\eta}\|_2 =
\|K_{\eta}\ast(K_h\ast(bf) - bf)\|_2
\leqslant
\|K\|_1\|bf - (bf)_h\|_2.
\end{displaymath}
On the other hand, let $\mathcal C$ be a countable and dense subset of the unit sphere of $\mathbb L^2(\mathbb R,dx)$ and consider $\mathfrak m(n) > 0$. Then, by Lemma \ref{preliminaries_GL},
\begin{eqnarray*}
 & &
 \mathbb E\left[
 \sup_{\eta\in\mathcal H_n}
 \left(\|\widehat{bf}_{n,\eta} -
 (bf)_{\eta}\|_{2}^{2} -\frac{V_n(\eta)}{6}\right)_+
 \right]\\
 & &
 \hspace{2cm}\leqslant
 \sum_{\eta\in\mathcal H_n}
 \mathbb E\left(\left(\sup_{\psi\in\mathcal C}
 2\mathbf V_{n,\eta}(\psi)^2 -\frac{V_n(\eta)}{6}
 \right)_+\right)
 + 2\sum_{\eta\in\mathcal H_n}\mathbb E(\mathbf W_{n,\eta})
\end{eqnarray*}
where, for any $\psi\in\mathcal C$,
\begin{displaymath}
\mathbf V_{n,\eta}(\psi) :=
\frac{1}{n}\sum_{k = 1}^{n}(v_{\psi,n,\eta}(X_k,Y_k) -\mathbb E(v_{\psi,n,\eta}(X_k,Y_k)))
\end{displaymath}
with
\begin{displaymath}
v_{\psi,n,\eta}(x,y) :=
y\mathbf 1_{|y|\leqslant\mathfrak m(n)}\int_{-\infty}^{\infty}\psi(u)K_{\eta}(x - u)du
\textrm{ $;$ }
\forall (x,y)\in\mathbb R^2,
\end{displaymath}
and
\begin{displaymath}
\mathbf W_{n,\eta} :=
\frac{1}{n}\sum_{k = 1}^{n}\int_{-\infty}^{\infty}|Y_k\mathbf 1_{|Y_k| >\mathfrak m(n)}K_{\eta}(X_k - u) -
\mathbb E(Y_k\mathbf 1_{|Y_k| >\mathfrak m(n)}K_{\eta}(X_k - u))|^2du.
\end{displaymath}
In order to apply Talagrand's inequality (see Klein and Rio \cite{KR05}), we compute bounds.
\begin{itemize}
 \item For every $\psi\in\mathcal C$, $x\in\mathbb R$ and $y\in [-\mathfrak m(n),\mathfrak m(n)]$,
 \begin{eqnarray*}
  |v_{\psi,n,\eta}(x,y)| & \leqslant &
  |y|\int_{-\infty}^{\infty}
  |\psi(u)|\cdot
  |K_{\eta}(u - x)|du\\
  & \leqslant &
  |y|\cdot\|K_{\eta}(\cdot - x)\|_2
  \leqslant
  \frac{\mathfrak m(n)\|K\|_2}{\sqrt{\eta}}.
 \end{eqnarray*}
 Then,
 \begin{displaymath}
 \sup_{\psi\in\mathcal C}\|v_{\psi,n,\eta}\|_{\infty}
 \leqslant\mathfrak m_1(n,\eta) :=
 \frac{\mathfrak m(n)\|K\|_2}{\sqrt{\eta}}.
 \end{displaymath}
 \item By Proposition \ref{bound_MISE_2bNW} and Lemma \ref{preliminaries_GL},
 \begin{eqnarray*}
  \mathbb E\left(\sup_{\psi\in\mathcal C}
  \mathbf V_{n,\eta}(\psi)^2\right)
  & \leqslant &
  \int_{-\infty}^{\infty}\textrm{var}(\widehat{bf}_{n,\eta}(u))du\\
  & \leqslant &
  \mathfrak m_{2}(n,\eta) :=
  \frac{\mathfrak c_{K,Y}}{n\eta}.
 \end{eqnarray*}
 \item For any $\psi\in\mathcal C$ and $k\in\{1,\dots,n\}$,
 \begin{eqnarray*}
  & &
  \textrm{var}(v_{\psi,n,\eta}(X_k,Y_k))\\
  & &
  \hspace{1cm}\leqslant
  \mathbb E\left(\left|Y_k\int_{-\infty}^{\infty}\psi(u)K_{\eta}(X_k - u)du\right|^2\right)\\
  & &
  \hspace{1cm}\leqslant
  \mathbb E((K_{\eta}\ast\psi)(X_1)^4)^{1/2}\mathbb E(Y_{1}^{4})^{1/2}
  \leqslant
  \|f\|_{\infty}^{1/2}
  \|K_{\eta}\ast\psi\|_{4}^{2}
  \mathbb E(Y_{1}^{4})^{1/2}.
 \end{eqnarray*}
 By Young's inequality, $\|K_{\eta}\ast\psi\|_4\leqslant\|\psi\|_2\|K_{\eta}\|_{4/3}$. So,
 \begin{displaymath}
 \textrm{var}(v_{\psi,n,\eta}(X_k,Y_k))
 \leqslant
 \mathfrak m_3 :=\frac{\mathfrak m_{f,K}}{\sqrt{\eta}}
 \end{displaymath}
 with $\mathfrak m_{f,K} :=\|f\|_{\infty}^{1/2}\|K\|_{4/3}^{2}\mathbb E(Y_{1}^{4})^{1/2}$.
\end{itemize}
By applying Talagrand's inequality to $(v_{\psi,n,\eta})_{\psi\in\mathcal C}$ and to the independent random variables $(X_1,Y_1),\dots,(X_n,Y_n)$, there exist three constants $\mathfrak c_1,\mathfrak c_2,\mathfrak c_3 > 0$, not depending on $n$ and $\eta$, such that
\begin{small}
\begin{eqnarray*}
 & &
 \mathbb E\left(\left(\sup_{\psi\in\mathcal C}\mathbf V_{n,\eta}(\psi)^2 - 4\mathfrak m_2(n,\eta)\right)_+\right)\\
 & &
 \hspace{0.75cm}\leqslant
 \mathfrak c_1\left(\frac{\mathfrak m_3}{n}\exp\left(-\mathfrak c_2\frac{n\mathfrak m_2(n,\eta)}{\mathfrak m_3}\right)
 +\frac{\mathfrak m_1(n,\eta)^2}{n^2}\exp\left(-\mathfrak c_3\frac{n\mathfrak m_2(n,\eta)^{1/2}}{\mathfrak m_1(n,\eta)}\right)\right)\\
 & &
 \hspace{0.75cm} =
 \mathfrak c_1\left[\frac{\mathfrak m_{f,K}}{n\sqrt{\eta}}
 \exp\left(-\frac{\mathfrak c_2\mathfrak c_{K,Y}}{\mathfrak m_{f,K}\sqrt{\eta}}\right)
 +\frac{1}{n^2\eta}\mathfrak m(n)^{2}\|K\|_{2}^{2}
 \exp\left(-\sqrt n\frac{\mathfrak c_3\mathbb E(Y_{1}^{2})^{1/2}}{\mathfrak m(n)}\right)\right].
\end{eqnarray*}
\end{small}
\newline
By taking $\mathfrak m(n) :=\mathfrak c_3\mathbb E(Y_{1}^{2})^{1/2}n^{1/2}/\log(n)^{1/2}$,
\begin{eqnarray*}
 & &
 \mathbb E\left(\left(\sup_{\psi\in\mathcal C}\mathbf V_{n,\eta}(\psi)^2 - 4\mathfrak m_2(n,\eta)\right)_+\right)\\
 & &
 \hspace{2cm}\leqslant
 \frac{\mathfrak c_1}{n}\left[\frac{\mathfrak m_{f,K}}{\sqrt{\eta}}
 \exp\left(-\frac{\mathfrak c_2\mathfrak c_{K,Y}}{\mathfrak m_{f,K}\sqrt{\eta}}\right)\right.
 \left.
 +\frac{\mathfrak c_{3}^{2}\mathbb E(Y_{1}^{2})\|K\|_{2}^{2}}{\eta n\log(n)}\right].
\end{eqnarray*}
By the conditional Markov inequality,
\begin{eqnarray*}
 \mathbb E(\mathbf W_{n,\eta}) & \leqslant &
 \int_{-\infty}^{\infty}\mathbb E(Y_{1}^{2}\mathbf 1_{|Y_1| >\mathfrak m(n)}K_{\eta}(X_1 - z)^2)dz\\
 & = &
 \frac{\|K\|_{2}^{2}}{\eta}\cdot
 \mathbb E(Y_{1}^{2}\mathbb E(\mathbf 1_{|Y_1| >\mathfrak m(n)}|Y_1))\\
 & \leqslant &
 \frac{\|K\|_{2}^{2}}{\eta\mathfrak m(n)^4}\mathbb E(Y_{1}^{6}) =
 \mathfrak c_{3}^{-4}\mathbb E(Y_{1}^{2})^{-2}\mathbb E(Y_{1}^{6})
 \|K\|_{2}^{2}\frac{\log(n)^2}{n^2\eta}.
\end{eqnarray*}
Finally, for $\upsilon\geqslant 48$,
\begin{displaymath}
\frac{V_n(\eta)}{12}
\geqslant 4\mathfrak m_2(n,\eta).
\end{displaymath}
Then, since
\begin{displaymath}
\frac{1}{n}
\sum_{\eta\in\mathcal H_n}\frac{1}{\eta}\leqslant\mathfrak m,
\quad\textrm{and}\quad
\sum_{\eta\in\mathcal H_n}
\frac{1}{\sqrt{\eta}}
\exp\left(-\frac{c}{\sqrt{\eta}}\right)
\leqslant\mathfrak m(c)
\textrm{ $;$ }
\forall c > 0,
\end{displaymath}
there exists a constant $\mathfrak c_4 > 0$, not depending on $n$, such that
\begin{eqnarray}
 & &
 \mathbb E\left[
 \sup_{\eta\in\mathcal H_n}
 \left(\|\widehat{bf}_{n,\eta} -
 (bf)_{\eta}\|_{2}^{2} -\frac{V_n(\eta)}{6}\right)_+
 \right]
 \nonumber\\
 & &
 \hspace{1cm}\leqslant
 2\sum_{\eta\in\mathcal H_n}
 \mathbb E\left[\left(\sup_{\psi\in\mathcal C}
 \mathbf V_{n,\eta}(\psi)^2 - 4\mathfrak m_2(n,\eta)
 \right)_+\right]
 +2\sum_{\eta\in\mathcal H_n}\mathbb E(\mathbf W_{n,\eta})
 \nonumber\\
 \label{bound_GL_3}
 & &
 \hspace{1cm}\leqslant
 \mathfrak c_4\frac{\log(n)^2}{n}
\end{eqnarray}
\newline
The same ideas give that there exists a constant $\mathfrak c_5 > 0$, not depending on $n$ and $h$, such that
\begin{equation}\label{bound_GL_4}
\mathbb E\left[
\sup_{\eta\in\mathcal H_n}
\left(\|\widehat{bf}_{n,h,\eta} -
(bf)_{h,\eta}\|_{2}^{2} -\frac{V_n(\eta)}{6}\right)_+
\right]
\leqslant
\mathfrak c_5
\frac{\log(n)^2}{n}
\end{equation}
Therefore, by Inequalities (\ref{bound_GL_1})--(\ref{bound_GL_4}), there exist two deterministic constants $\mathfrak c,\overline{\mathfrak c} > 0$, not depending on $n$, such that
\begin{displaymath}
\mathbb E(\|\widehat{bf}_{n,\widehat h_n} - bf\|_{2}^{2})
\leqslant
\mathfrak c\cdot\inf_{h\in\mathcal H_n}\{\|(bf)_h - bf\|_{2}^{2} + V_n(h)\} +\overline{\mathfrak c}\frac{\log(n)^2}{n}.
\end{displaymath}
%


%
\subsection{Proof of Corollary \ref{bound_2bNW_GL}}\label{proof_bound_2bNW_GL}
As established in the proof of Proposition \ref{bound_MISE_2bNW},
\begin{eqnarray*}
 & &
 \|\widehat b_{n,\widehat h_n,\widehat h_n'} - b\|_{2,f,\mathcal S_n}^{2}\\
 & &
 \hspace{1cm}\leqslant
 \frac{8\mathfrak c_1}{m_{n}^{2}}\left(\|\widehat{bf}_{n,\widehat h_n} - bf\|_{2}^{2}
 +\mathfrak c_{b,f}\int_{-\infty}^{\infty}|\widehat f_{n,\widehat h_n'}(x) - f(x)|^2dx
 \right)\\
 & &
 \hspace{2cm}
 + 2\mathfrak c_{b,f}\int_{\mathcal S_n}
 \mathbf 1_{|\widehat f_{n,\widehat h_n'}(x) - f(x)| > m_n/2}dx
\end{eqnarray*}
with $\mathfrak c_1 :=\|f\|_{\infty}\vee\|f\|_{\infty}^{2}$. By Markov's inequality,
\begin{eqnarray*}
 \mathbb E(\|\widehat b_{n,\widehat h_n,\widehat h_n'} - b\|_{2,f,\mathcal S_n}^{2})
 & \leqslant &
 \frac{8\mathfrak c_1}{m_{n}^{2}}(\mathbb E(\|\widehat{bf}_{n,\widehat h_n} - bf\|_{2}^{2}) +
 \mathfrak c_{b,f}
 \mathbb E(\|\widehat f_{n,\widehat h_n'} - f\|_{2}^{2}))\\
 & &
 + 2\mathfrak c_{b,f}\int_{\mathcal S_n}
 \mathbb P\left(|\widehat f_{n,\widehat h_n'}(x) - f(x)| >\frac{m_n}{2}\right)dx\\
 & \leqslant &
 \frac{8(\mathfrak c_1\vee 1)}{m_{n}^{2}}(\mathbb E(\|\widehat{bf}_{n,\widehat h_n} - bf\|_{2}^{2})\\
 & &
 \hspace{3cm} +
 2\mathfrak c_{b,f}\mathbb E(\|\widehat f_{n,\widehat h_n'} - f\|_{2}^{2})).
\end{eqnarray*}
Theorem \ref{bound_GL} and Inequality (\ref{GL_density}) allow to conclude.
%


%
\subsection{Proof of Theorem \ref{bound_LMR}}
The proof relies on three lemmas, which are stated first.
%


%
\begin{lemma}\label{bound_U}
Consider the $U$-statistic
\begin{displaymath}
U_n(h,h_{\min}) :=
\sum_{k\not= l}
\langle Y_kK_h(X_k -\cdot) - (bf)_h,Y_lK_{h_{\min}}(X_l -\cdot) - (bf)_{h_{\min}}\rangle_2.
\end{displaymath}
Under Assumption \ref{assumption_K_f_numerator}, if there exists $\alpha > 0$ such that $\mathbb E(\exp(\alpha|Y_1|)) <\infty$, then there exists a deterministic constant $\mathfrak c_U > 0$, not depending on $n$ and $h_{\min}$, such that for every $\vartheta\in (0,1)$,
\begin{displaymath}
\mathbb E\left(\sup_{h\in\mathcal H_n}\left\{\frac{|U_n(h,h_{\min})|}{n^2}
-\frac{\vartheta\|K\|_{2}^{2}}{nh}\mathbb E(Y_{1}^{2})\right\}\right)
\leqslant
\mathfrak c_U\frac{\log(n)^5}{\vartheta n}.
\end{displaymath}
\end{lemma}
%


%
\begin{lemma}\label{bound_V}
For every $\eta,\eta'\in\mathcal H_n$, consider
\begin{displaymath}
V_n(\eta,\eta') :=
\langle\widehat{bf}_{n,\eta} - (bf)_{\eta'},
(bf)_{\eta'} - bf\rangle_2.
\end{displaymath}
Under Assumption \ref{assumption_K_f_numerator}, if there exists $\alpha > 0$ such that $\mathbb E(\exp(\alpha|Y_1|)) <\infty$ and $bf$ is bounded, then there exists a deterministic constant $\mathfrak c_V > 0$, not depending on $n$ and $h_{\min}$, such that for every $\vartheta\in (0,1)$,
\begin{displaymath}
\mathbb E\left(\sup_{\eta,\eta'\in\mathcal H_n}\{
|V_n(\eta,\eta')| -\vartheta\|(bf)_{\eta'} - bf\|_{2}^{2}\}\right)
\leqslant
\mathfrak c_V\frac{\log(n)^3}{\vartheta n}.
\end{displaymath}
\end{lemma}
%


%
\begin{lemma}\label{Lerasle_type_inequality}
Under Assumption \ref{assumption_K_f_numerator}, if $bf$ is bounded and if there exists $\alpha > 0$ such that $\mathbb E(\exp(\alpha|Y_1|)) <\infty$, then there exists a deterministic constant $\mathfrak c_L > 0$, not depending on $n$ and $h_{\min}$, such that for every $\vartheta\in (0,1)$,
\begin{displaymath}
\mathbb E\left(
\sup_{h\in\mathcal H_n}\left\{
\|(bf)_h - bf\|_{2}^{2} +\frac{\mathfrak c_{K,Y}}{nh} -\frac{1}{1 -\vartheta}
\|\widehat{bf}_{n,h} - bf\|_{2}^{2}\right\}\right)
\leqslant
\frac{\mathfrak c_L}{\vartheta(1 -\vartheta)}
\cdot\frac{\log(n)^5}{n}.
\end{displaymath}
\end{lemma}
%


%
\subsubsection{Steps of the proof.} The proof of Theorem \ref{bound_LMR} is dissected in three steps.
\\
\\
\textbf{Step 1.} In this step, a suitable decomposition of
\begin{displaymath}
\|\widehat{bf}_{n,\widetilde h_n} - bf\|_{2}^{2}
\end{displaymath}
is provided. On the one hand,
\begin{eqnarray*}
 & &
 \|\widehat{bf}_{n,\widetilde h_n} - bf\|_{2}^{2} +\textrm{pen}(\widetilde h_n)\\
 & &
 \hspace{1cm} =
 \|\widehat{bf}_{n,\widetilde h_n} -\widehat{bf}_{n,h_{\min}}\|_{2}^{2} +\textrm{pen}(\widetilde h_n)\\
 & &
 \hspace{2cm}
 +\|\widehat{bf}_{n,h_{\min}} - bf\|_{2}^{2}
 - 2\langle\widehat{bf}_{n,h_{\min}} -\widehat{bf}_{n,\widetilde h_n},\widehat{bf}_{n,h_{\min}} - bf\rangle_2.
\end{eqnarray*}
Since
\begin{displaymath}
\widetilde h_n\in\arg\min_{h\in\mathcal H_n}\textrm{crit}(h)
\quad{\rm with}\quad
\textrm{crit}(h) =\|\widehat{bf}_{n,h} -\widehat{bf}_{n,h_{\min}}\|_{2}^{2} +\textrm{pen}(h),
\end{displaymath}
for any $h\in\mathcal H_n$,
\begin{equation}\label{bound_LMR_1}
\|\widehat{bf}_{n,\widetilde h_n} - bf\|_{2}^{2}
\leqslant
\|\widehat{bf}_{n,h} - bf\|_{2}^{2} +\textrm{pen}(h) - 2\psi_n(h) -
(\textrm{pen}(\widetilde h_n) - 2\psi_n(\widetilde h_n))
\end{equation}
with
\begin{displaymath}
\psi_n(h) :=
\langle\widehat{bf}_{n,h_{\min}} - bf,\widehat{bf}_{n,h} - bf\rangle_2.
\end{displaymath}
On the other hand,
\begin{displaymath}
\psi_n(h) =
\psi_{1,n}(h) +
\psi_{2,n}(h) +
\psi_{3,n}(h)
\end{displaymath}
where
\begin{eqnarray*}
 \psi_{1,n}(h) & := &
 \frac{\langle K_{h_{\min}},K_h\rangle_2}{n^2}\sum_{k = 1}^{n}Y_{k}^{2} +\frac{U_n(h,h_{\min})}{n^2},\\
 \psi_{2,n}(h) & := &
 -\frac{1}{n^2}\left(\sum_{k = 1}^{n}Y_k\langle K_{h_{\min}}(X_k -\cdot),(bf)_h\rangle_2\right.\\
 & &
 \left. +
 \sum_{k = 1}^{n}Y_k\langle K_h(X_k -\cdot),(bf)_{h_{\min}}\rangle_2\right) +\frac{1}{n}\langle (bf)_{h_{\min}},(bf)_h\rangle_2\textrm{ and}\\
 \psi_{3,n}(h) & := &
 V_n(h,h_{\min}) + V_n(h_{\min},h) +\langle (bf)_h - bf,(bf)_{h_{\min}} - bf\rangle_2.
\end{eqnarray*}
\textbf{Step 2.} In this step, let us provide some suitable controls of
\begin{displaymath}
\mathbb E(\psi_{i,n}(h))
\textrm{ and }
\mathbb E(\psi_{i,n}(\widetilde h_n))
\textrm{ $;$ }i = 1,2,3.
\end{displaymath}
\begin{enumerate}
 \item Consider
 \begin{displaymath}
 \widetilde\psi_{1,n}(h) :=
 \psi_{1,n}(h) -\frac{\langle K_{h_{\min}},K_h\rangle_2}{n^2}\sum_{k = 1}^{n}Y_{k}^{2} = \frac{U(h, h_{\min})}{n^2}.
 \end{displaymath}
 By Lemma \ref{bound_U},
 \begin{displaymath}
 \mathbb E(|\widetilde\psi_{1,n}(h)|)
 \leqslant
 \frac{\theta\|K\|_{2}^{2}}{nh}\mathbb E(Y_{1}^{2}) +
 \frac{2\mathfrak c_U}{\theta}\cdot\frac{\log(n)^5}{n}
 \end{displaymath}
 and
 \begin{displaymath}
 \mathbb E(|\widetilde\psi_{1,n}(\widetilde h_n)|)
 \leqslant
 \mathbb E\left(\frac{\theta\|K\|_{2}^{2}}{n\widetilde h_n}\right)\mathbb E(Y_{1}^{2}) +
 \frac{2\mathfrak c_U}{\theta}\cdot\frac{\log(n)^5}{n}.
 \end{displaymath}
 \item On the one hand, for every $\eta,\eta'\in\mathcal H_n$, consider
 \begin{displaymath}
 \Psi_{2,n}(\eta,\eta') :=
 \frac{1}{n}\sum_{k = 1}^{n}Y_k\langle K_{\eta}(X_k -\cdot),(bf)_{\eta'}\rangle_2.
 \end{displaymath}
 Then,
 \begin{eqnarray*}
  & &
  \mathbb E\left(\sup_{\eta,\eta'\in\mathcal H_n}|\Psi_{2,n}(\eta,\eta')|\right)\\
  & &
  \hspace{1cm}\leqslant
  \mathbb E\left(|Y_1|\sup_{\eta,\eta'\in\mathcal H_n}\int_{-\infty}^{\infty}
  |K_{\eta}(X_1 - u)(bf)_{\eta'}(u)|du\right)\\
  & &
  \hspace{1cm}\leqslant
  \mathbb E(Y_{1}^{2})^{1/2}
  \|K\|_{1}^{2}\|bf\|_{\infty}.
 \end{eqnarray*}
 On the other hand,
 \begin{eqnarray*}
  \sup_{\eta,\eta'\in\mathcal H_n}
  |\langle (bf)_{\eta},(bf)_{\eta'}\rangle_2|
  & \leqslant &
  \sup_{\eta,\eta'\in\mathcal H_n}
  \|K_{\eta}\ast (bf)\|_2
  \|K_{\eta'}\ast (bf)\|_2\\
  & \leqslant &
  \|K\|_{1}^{2}\|bf\|_{2}^{2}
  \leqslant
  \mathbb E(Y_{1}^{2})^{1/2}
  \|K\|_{1}^{2}\|bf\|_{\infty}.
 \end{eqnarray*}
 Then,
 \begin{displaymath}
 \mathbb E(|\psi_{2,n}(h)|)
 \leqslant
 \frac{3}{n}\mathbb E(Y_{1}^{2})^{1/2}
  \|K\|_{1}^{2}\|bf\|_{\infty}
 \end{displaymath}
 and
 \begin{displaymath}
 \mathbb E(|\psi_{2,n}(\widetilde h_n)|)
 \leqslant
 \frac{3}{n}\mathbb E(Y_{1}^{2})^{1/2}
  \|K\|_{1}^{2}\|bf\|_{\infty}.
 \end{displaymath}
 \item By Lemma \ref{bound_V},
 \begin{eqnarray*}
  \mathbb E(|\psi_{n,3}(h)|) & \leqslant &
  \frac{\theta}{2}
  (\|(bf)_h - bf\|_{2}^{2} +\|(bf)_{h_{\min}} - bf\|_{2}^{2}) +
  4\mathfrak c_V\frac{\log(n)^3}{\theta n}\\
  & & +
  \left(\frac{\theta}{2}\right)^{1/2}\|(bf)_h - bf\|_2\times
  \left(\frac{2}{\theta}\right)^{1/2}\|(bf)_{h_{\min}} - bf\|_2\\
  & \leqslant &
  \theta\|(bf)_h - bf\|_{2}^{2} +
  \left(\frac{\theta}{2} +\frac{2}{\theta}\right)\|(bf)_{h_{\min}} - bf\|_{2}^{2}\\
  & &
  + 4\mathfrak c_V\frac{\log(n)^3}{\theta n}
 \end{eqnarray*}
 and
 \begin{eqnarray*}
  \mathbb E(|\psi_{n,3}(\widetilde h_n)|) & \leqslant &
  \theta\mathbb E(\|(bf)_{\widetilde h_n} - bf\|_{2}^{2})\\
  & & +
  \left(\frac{\theta}{2} +\frac{2}{\theta}\right)\|(bf)_{h_{\min}} - bf\|_{2}^{2}
  + 4\mathfrak c_V\frac{\log(n)^3}{\theta n}.
 \end{eqnarray*}
\end{enumerate}
\textbf{Step 3.} Consider
\begin{displaymath}
\widetilde\psi_n(h) :=
\psi_n(h)
-\frac{\langle K_{h_{\min}},K_h\rangle_2}{n^2}\sum_{k = 1}^{n}Y_{k}^{2}.
\end{displaymath}
By Step 2, there exists a deterministic constant $\mathfrak c_{U,V} > 0$, not depending on $n$, $h$ and $h_{\min}$, such that
\begin{eqnarray*}
 \mathbb E(|\widetilde\psi_n(h)|)
 & \leqslant &
 \theta\left(\|(bf)_h - bf\|_{2}^{2} +
 \frac{\mathfrak c_{K,Y}}{nh}\right)\\
 & &
 +\frac{\mathfrak c_{U,V}}{\theta}\cdot\frac{\log(n)^5}{n}
 +\left(\frac{\theta}{2} +\frac{2}{\theta}\right)\|(bf)_{h_{\min}} - bf\|_{2}^{2}
\end{eqnarray*}
and
\begin{eqnarray*}
 \mathbb E(|\widetilde\psi_n(\widetilde h_n)|)
 & \leqslant &
 \theta\left[\mathbb E(\|(bf)_{\widetilde h_n} - bf\|_{2}^{2}) +
 \mathbb E\left(\frac{\mathfrak c_{K,Y}}{n\widetilde h_n}\right)\right]\\
 & &
 +\frac{\mathfrak c_{U,V}}{\theta}\cdot\frac{\log(n)^5}{n}
 +\left(\frac{\theta}{2} +\frac{2}{\theta}\right)\|(bf)_{h_{\min}} - bf\|_{2}^{2}.
\end{eqnarray*}
Then, by Lemma \ref{Lerasle_type_inequality},
\begin{eqnarray*}
 \mathbb E(|\widetilde\psi_n(h)|)
 & \leqslant &
 \frac{\theta}{1 -\theta}\mathbb E(\|\widehat{bf}_{n,h} - bf\|_{2}^{2})
 +\left(\frac{\theta}{2} +\frac{2}{\theta}\right)\|(bf)_{h_{\min}} - bf\|_{2}^{2}\\
 & &
 +\left(\frac{\mathfrak c_{U,V}}{\theta} +\frac{\mathfrak c_L}{1 -\theta}\right)
 \frac{\log(n)^5}{n}
\end{eqnarray*}
and
\begin{eqnarray*}
 \mathbb E(|\widetilde\psi_n(\widetilde h_n)|)
 & \leqslant &
 \frac{\theta}{1 -\theta}\mathbb E(\|\widehat{bf}_{n,\widetilde h_n} - bf\|_{2}^{2})
 +\left(\frac{\theta}{2} +\frac{2}{\theta}\right)\|(bf)_{h_{\min}} - bf\|_{2}^{2}\\
 & &
 +\left(\frac{\mathfrak c_{U,V}}{\theta} +\frac{\mathfrak c_L}{1 -\theta}\right)
 \frac{\log(n)^5}{n}.
\end{eqnarray*}
By Inequality (\ref{bound_LMR_1}), there exist two deterministic constant $\mathfrak c_1,\mathfrak c_2 > 0$, not depending on $n$, $h$ and $h_{\min}$, such that
\begin{eqnarray*}
 \mathbb E(\|\widehat{bf}_{n,\widetilde h_n} - bf\|_{2}^{2})
 & \leqslant &
 \mathbb E(\|\widehat{bf}_{n,h} - bf\|_{2}^{2})
 + 2(\mathbb E(|\widetilde\psi_n(h)|)
 +\mathbb E(|\widetilde\psi_n(\widetilde h_n)|))\\
 & \leqslant &
 \left(1 +\frac{2\theta}{1 -\theta}\right)\mathbb E(\|\widehat{bf}_{n,h} - bf\|_{2}^{2})\\
 & & +
 \frac{2\theta}{1 -\theta}\mathbb E(\|\widehat{bf}_{n,\widetilde h_n} - bf\|_{2}^{2})\\
 & & +
 \frac{\mathfrak c_1}{\theta}\|(bf)_{h_{\min}} - bf\|_{2}^{2} +
 \frac{\mathfrak c_2}{\theta(1 -\theta)}\cdot\frac{\log(n)^5}{n}.
\end{eqnarray*}
This concludes the proof.
%


%
\subsubsection{Proof of Lemma \ref{bound_U}}
Consider
\begin{displaymath}
\Delta_n :=
\{(k,l)\in\{1,\dots,n\} : 2\leqslant k\textrm{ and }l < k\}
\end{displaymath}
and $Z_k := (X_k,Y_k)$ for every $k\in\{1,\dots,n\}$.
\\
\\
On the one hand, consider $n\in\mathbb N$ such that $\mathfrak m(n) := 4\log(n)/\alpha\geqslant 1$ and
\begin{displaymath}
U_{1,n}(h,h_{\min}) :=
\sum_{k = 2}^{n}
\sum_{l < k}
(G_{n,h,h_{\min}}(Z_k,Z_l) +
G_{n,h_{\min},h}(Z_k,Z_l))
\end{displaymath}
where, for every $\eta,\eta'\in\{h,h_{\min}\}$ and $z,z'\in\mathbb R^2$,
\begin{eqnarray*}
 & &
 G_{n,\eta,\eta'}(z,z') :=\\
 & &
 \hspace{0.75cm}
 \langle z_2\mathbf 1_{|z_2|\leqslant\mathfrak m(n)}K_{\eta}(z_1 -\cdot) - (bf)_{n,\eta},
 z_2'\mathbf 1_{|z_2'|\leqslant\mathfrak m(n)}K_{\eta'}(z_1' -\cdot) - (bf)_{n,\eta'}\rangle_2
\end{eqnarray*}
and
\begin{displaymath}
(bf)_{n,\eta} :=\mathbb E(Y_1\mathbf 1_{|Y_1|\leqslant\mathfrak m(n)}K_{\eta}(X_1 -\cdot)).
\end{displaymath}
For every $\eta,\eta'\in\{h,h_{\min}\}$ and $(k,l)\in\Delta_n$,
\begin{eqnarray*}
 \mathbb E(G_{n,\eta,\eta'}(Z_k,Z_l)|Z_k) & = &
 \int_{-\infty}^{\infty}(Y_k\mathbf 1_{|Y_k|\leqslant\mathfrak m(n)}K_{\eta}(X_k - z) - (bf)_{n,\eta}(z))\\
 & &
 \times\mathbb E(Y_l\mathbf 1_{|Y_l|\leqslant\mathfrak m(n)}K_{\eta'}(X_l - z) - (bf)_{n,\eta'}(z))dz
 = 0.
\end{eqnarray*}
So, by Houdr\'e and Reynaud-Bouret \cite{HRB03}, Theorem 3.4, there exists a universal constant $\mathfrak e > 0$ such that
\begin{equation}\label{bound_U_1}
\mathbb P(|U_{1,n}(h,h_{\min})|\geqslant\mathfrak e(\mathfrak c_n\lambda^{1/2} +\mathfrak d_n\lambda +\mathfrak b_n\lambda^{3/2} +\mathfrak a_n\lambda^2))\leqslant 5.54e^{-\lambda}
\end{equation}
where the constants $\mathfrak a_n$, $\mathfrak b_n$, $\mathfrak c_n$ and $\mathfrak d_n$ will be defined and controlled in the sequel.
\begin{itemize}
 \item\textbf{The constant $\mathfrak a_n$.} Consider
 \begin{displaymath}
 \mathfrak a_n :=
 \sup_{(z,z')\in\mathbb R^2\times\mathbb R^2}
 \mathbf A_n(z,z'),
 \end{displaymath}
 where
 \begin{displaymath}
 \mathbf A_n(z,z') :=
 |G_{n,h,h_{\min}}(z,z') + G_{n,h_{\min},h}(z,z')|
 \textrm{ $;$ }
 \forall z,z'\in\mathbb R^2.
 \end{displaymath}
 First, note that for every $\eta\in\mathcal H_n$,
 \begin{displaymath}
 \|(bf)_{n,\eta}\|_1
 \leqslant\mathbb E(|Y_1|\mathbf 1_{|Y_1|\leqslant\mathfrak m(n)})\|K\|_1
 \leqslant\mathfrak m(n)\|K\|_1
 \end{displaymath}
 and
 \begin{displaymath}
 \|(bf)_{n,\eta}\|_{\infty}
 \leqslant\frac{\mathfrak m(n)\|K\|_{\infty}}{\eta}.
 \end{displaymath}
 For any $z,z'\in\mathbb R\times [-\mathfrak m(n),\mathfrak m(n)]$,
 \begin{eqnarray*}
  \mathbf A_n(z,z') & \leqslant &
  \langle z_2K_h(z_1 -\cdot) - (bf)_{n,h},z_2'K_{h_{\min}}(z_1' -\cdot) - (bf)_{n,h_{\min}}\rangle_2\\
  & & +
  \langle z_2K_{h_{\min}}(z_1 -\cdot) - (bf)_{n,h_{\min}},z_2'K_h(z_1' -\cdot) - (bf)_{n,h}\rangle_2\\
  & \leqslant &
  2(\mathfrak m(n)\|K_{h_{\min}}\|_{\infty} +\|(bf)_{n,h_{\min}}\|_{\infty})\\
  & &
  \times(\mathfrak m(n)\|K\|_1 +\|(bf)_{n,h}\|_1)\\
  & \leqslant &
  \frac{8\|K\|_1\|K\|_{\infty}}{h_{\min}}\mathfrak m(n)^2.
 \end{eqnarray*}
 Therefore,
 \begin{displaymath}
 \frac{\mathfrak a_n\lambda^2}{n^2}
 \leqslant\frac{8\|K\|_1\|K\|_{\infty}}{n^2h_{\min}}\mathfrak m(n)^2\lambda^2.
 \end{displaymath}
 \item\textbf{The constant $\mathfrak b_n$.} Consider
 \begin{displaymath}
 \mathfrak b_{n}^{2} :=
 n\max\left\{
 \sup_{z\in\mathbb R^2}\mathbb E(G_{n,h,h_{\min}}(z,Z_1)^2)
 \textrm{ $;$ }
 \sup_{z\in\mathbb R^2}\mathbb E(G_{n,h_{\min},h}(z,Z_1)^2)
 \right\}.
 \end{displaymath}
 First, note that for every $\eta\in\mathcal H_n$,
 \begin{displaymath}
 \|(bf)_{n,\eta}\|_{2}^{2}
 \leqslant\frac{\mathfrak m(n)^2\|K\|_{2}^{2}}{\eta}.
 \end{displaymath}
 For any $\eta,\eta'\in\{h,h_{\min}\}$ and $z\in\mathbb R\times [-\mathfrak m(n),\mathfrak m(n)]$,
 \begin{eqnarray*}
  & &
  \mathbb E(G_{n,\eta,\eta'}(z,Z_1)^2)\\
  & &
  \hspace{1cm}\leqslant
  \|z_2K_{\eta}(z_1 -\cdot) - (bf)_{n,\eta}\|_{2}^{2}\\
  & &
  \hspace{2cm}
  \times\int_{-\infty}^{\infty}
  \mathbb E(|Y_1\mathbf 1_{|Y_1|\leqslant\mathfrak m(n)}K_{\eta'}(X_1 - u) - (bf)_{n,\eta'}(u)|^2)du\\
  & &
  \hspace{1cm}\leqslant
  \frac{2\|K\|_{2}^{2}}{\eta}\mathfrak m(n)^2
  \int_{-\infty}^{\infty}\textrm{var}(Y_1\mathbf 1_{|Y_1|\leqslant\mathfrak m(n)}K_{\eta'}(X_1 - u))du\\
  & &
  \hspace{1cm}\leqslant
  \frac{2\|K\|_{2}^{4}}{\eta\eta'}\mathbb E(Y_{1}^{2})\mathfrak m(n)^2.
 \end{eqnarray*}
 Therefore, for any $\theta\in (0,1)$,
 \begin{eqnarray*}
  \frac{\mathfrak b_n\lambda^{3/2}}{n^2}
  & \leqslant &
  \sqrt 2\cdot\frac{\|K\|_{2}^{2}}{n^{3/2}(hh_{\min})^{1/2}}\mathbb E(Y_{1}^{2})^{1/2}\mathfrak m(n)\lambda^{3/2}\\
  & \leqslant &
  2\left(\frac{3\mathfrak e}{\theta}\right)^{1/2}\frac{\|K\|_2}{nh_{\min}^{1/2}}\mathfrak m(n)\lambda^{3/2}
  \times\left(\frac{\theta}{3\mathfrak e}\right)^{1/2}\frac{\|K\|_2}{n^{1/2}h^{1/2}}\mathbb E(Y_{1}^{2})^{1/2}\\
  & \leqslant &
  \frac{3\mathfrak e\|K\|_{2}^{2}}{\theta n^2h_{\min}}\mathfrak m(n)^2\lambda^3 
  +\frac{\theta\|K\|_{2}^{2}}{3\mathfrak enh}\mathbb E(Y_{1}^{2}).
 \end{eqnarray*}
 \item\textbf{The constant $\mathfrak c_n$.} Consider
 \begin{displaymath}
 \mathfrak c_{n}^{2} :=
 \sum_{(k,l)\in\Delta_n}\mathbb E(|G_{n,h,h_{\min}}(Z_k,Z_l) + G_{n,h_{\min},h}(Z_k,Z_l)|^2).
 \end{displaymath}
 First, note that for every $\eta\in\mathcal H_n$,
 \begin{displaymath}
 \|(bf)_{n,\eta}\|_{\infty}
 \leqslant
 \mathfrak m(n)\|f\|_{\infty}\|K\|_1.
 \end{displaymath}
 For any $\eta,\eta'\in\{h,h_{\min}\}$ and $(k,l)\in\Delta_n$,
 \begin{eqnarray*}
  \mathbb E(G_{n,\eta,\eta'}(Z_k,Z_l)^2) & \leqslant &
  4(\mathfrak m(n)^2\mathbb E(\langle K_{\eta}(X_k -\cdot),K_{\eta'}(X_l -\cdot)\rangle_{2}^{2}Y_{l}^{2})\\
  & &
  +\|(bf)_{n,\eta}\|_{\infty}^{2}\mathbb E(Y_{l}^{2}\|K_{\eta'}(X_l -\cdot)\|_{1}^{2})\\
  & &
  +\|(bf)_{n,\eta'}\|_{\infty}^{2}\mathbb E(Y_{k}^{2}\|K_{\eta}(X_k -\cdot)\|_{1}^{2})\\
  & &
  +\|(bf)_{n,\eta}\|_{\infty}^{2}\|(bf)_{n,\eta'}\|_{1}^{2})\\
  & \leqslant &
  4\mathfrak m(n)^2(\mathbb E(\langle K_{\eta}(X_k -\cdot),K_{\eta'}(X_l -\cdot)\rangle_{2}^{2}Y_{l}^{2})\\
  & &
  + 3\|f\|_{\infty}^{2}\|K\|_{1}^{4}\mathbb E(Y_{1}^{2})).
 \end{eqnarray*}
 Moreover,
 \begin{eqnarray*}
  \mathbb E(\langle K_{\eta}(X_k -\cdot),K_{\eta'}(X_l -\cdot)\rangle_{2}^{2}Y_{l}^{2}) & = &
  \sigma^2\mathbb E((K_{\eta}\ast K_{\eta'})(X_k - X_l)^2)\\
  & & +
  \mathbb E((K_{\eta}\ast K_{\eta'})(X_k - X_l)^2b(X_l)^2)\\
  & \leqslant &
  \sigma^2\|f\|_{\infty}\|K_{\eta}\ast K_{\eta'}\|_{2}^{2}\\
  & &
  +\|f\|_{\infty}\mathbb E(b(X_1)^2)\|K_{\eta}\ast K_{\eta'}\|_{2}^{2}\\
  & \leqslant &
  \frac{\|f\|_{\infty}\|K\|_{1}^{2}\|K\|_{2}^{2}}{\eta}\mathbb E(Y_{1}^{2}).
 \end{eqnarray*}
 Then, there exists a universal constant $\mathfrak c_1 > 0$ such that
 \begin{displaymath}
 \mathfrak c_{n}^{2}
 \leqslant
 \mathfrak c_1n^2\|f\|_{\infty}\|K\|_{1}^{2}\mathfrak m(n)^2\mathbb E(Y_{1}^{2})\left(
 \frac{\|K\|_{2}^{2}}{h} +
 3\|f\|_{\infty}\|K\|_{1}^{2}\right).
 \end{displaymath}
 Therefore, since $\mathfrak m(n)$ is larger than $1$, there exists a universal constant $\mathfrak c_2 > 0$ such that
 \begin{displaymath}
 \frac{\mathfrak c_n\lambda^{1/2}}{n^2}
 \leqslant
 \frac{\theta\|K\|_{2}^{2}}{3\mathfrak enh}\mathbb E(Y_{1}^{2}) +
 \frac{\mathfrak c_2}{n\theta}\|f\|_{\infty}\|K\|_{1}^{2}
 \mathfrak m(n)^2(\lambda^{1/2} +\lambda).
 \end{displaymath}
 \item\textbf{The constant $\mathfrak d_n$.} Consider
 \begin{displaymath}
 \mathfrak d_n :=
 \sup_{(\alpha,\beta)\in\mathcal S}
 \sum_{(k,l)\in\Delta_n}
 \mathbb E((G_{h,h_{\min}}(Z_k,Z_l) + G_{h_{\min},h}(Z_k,Z_l))\alpha_k(Z_k)\beta_l(Z_l)),
 \end{displaymath}
 where
 \begin{displaymath}
 \mathcal S :=
 \left\{
 (\alpha,\beta) :
 \sum_{k = 2}^{n}\mathbb E(\alpha_k(Z_k)^2)\leqslant 1
 \textrm{ and }
 \sum_{l = 1}^{n - 1}\mathbb E(\beta_l(Z_l)^2)\leqslant 1
 \right\}.
 \end{displaymath}
 For any $(\alpha,\beta)\in\mathcal S$,
 \begin{displaymath}
 \sum_{(k,l)\in\Delta_n}
 \mathbb E(G_{h,h_{\min}}(Z_k,Z_l)\alpha_k(Z_k)\beta_l(Z_l))
 \leqslant
 \mathbf D_2(\alpha,\beta)\sup_{u\in\mathbb R}\mathbf D_1(\alpha,\beta,u)
 \end{displaymath}
 with, for every $u\in\mathbb R$,
 \begin{eqnarray*}
  \mathbf D_1(\alpha,\beta,u)
  & := &
  \sum_{k = 2}^{n}\mathbb E(|\alpha_k(Z_k)(Y_k\mathbf 1_{|Y_k|\leqslant\mathfrak m(n)}K_h(X_k - u) - (bf)_{n,h}(u))|)\\
  & \leqslant &
  \mathbb E\left[\left(\sum_{k = 2}^{n}\alpha_k(Z_k)^2\right)^{1/2}\right.\\
  & &
  \left.\times
  \left(\sum_{k = 2}^{n}|Y_k\mathbf 1_{|Y_k|\leqslant\mathfrak m(n)}K_h(X_k - u) - (bf)_{n,h}(u)|^2\right)^{1/2}\right]\\
  & \leqslant &
  \left(\sum_{k = 2}^{n}\mathbb E(\alpha_k(Z_k)^2)\right)^{1/2}\\
  & &
  \times
  \left(\sum_{k = 2}^{n}\mathbb E(Y_{k}^{2}\mathfrak 1_{|Y_k|\leqslant\mathfrak m(n)}K_h(X_k - u)^2)\right)^{1/2}\\
  & \leqslant &
  \frac{\|f\|_{\infty}^{1/2}\|K\|_2}{h^{1/2}}n^{1/2}\mathfrak m(n)
 \end{eqnarray*}
 and
 \begin{small}
 \begin{eqnarray*}
  & &
  \mathbf D_2(\alpha,\beta)\\
  & &
  \hspace{0.2cm} :=
  \sum_{l = 1}^{n - 1}\mathbb E\left(|\beta_l(Z_l)|\int_{-\infty}^{\infty}
  |Y_l\mathbf 1_{|Y_l|\leqslant\mathfrak m(n)}K_{h_{\min}}(X_l - u) - (bf)_{n,h_{\min}}(u)|du\right)\\
  & &
  \hspace{0.2cm}\leqslant
  \sqrt 2\left(\sum_{l = 1}^{n - 1}\mathbb E(\beta_l(Z_l)^2)\right)^{1/2}\\
  & &
  \hspace{1cm}\times
  \left(
  \sum_{l = 1}^{n - 1}[\mathbb E(
  Y_{l}^{2}\|K_{h_{\min}}(X_l -\cdot)\|_{1}^{2}) +\|(bf)_{n,h_{\min}}\|_{1}^{2}]\right)^{1/2}\\
  & &
  \hspace{0.2cm}\leqslant
  \sqrt 2\cdot\|K\|_1\mathbb E(Y_{1}^{2})^{1/2}n^{1/2}.
 \end{eqnarray*}
 \end{small}
 \newline
 Then,
 \begin{displaymath}
 \mathfrak d_n\leqslant
 2n\frac{\|K\|_2\|K\|_1\|f\|_{\infty}^{1/2}}{h^{1/2}}\mathbb E(Y_{1}^{2})^{1/2}\mathfrak m(n).
 \end{displaymath}
 Therefore,
 \begin{eqnarray*}
  \frac{\mathfrak d_n\lambda}{n^2}
  & \leqslant &
  2\times\left(\frac{\theta}{3\mathfrak e}\right)^{1/2}\frac{\|K\|_2}{(nh)^{1/2}}\mathbb E(Y_{1}^{2})^{1/2}
  \times\left(\frac{3\mathfrak e}{\theta}\right)^{1/2}\frac{\|K\|_1\|f\|_{\infty}^{1/2}}{n^{1/2}}\mathfrak m(n)\lambda\\
  & \leqslant &
  \frac{\theta\|K\|_{2}^{2}}{3\mathfrak enh}\mathbb E(Y_{1}^{2}) +
  \frac{3\mathfrak e\|K\|_{1}^{2}\|f\|_{\infty}}{\theta n}\mathfrak m(n)^2\lambda^2.
 \end{eqnarray*}
\end{itemize}
So, there exist two universal constants $\mathfrak c_3,\mathfrak c_4 > 0$ such that, with probability larger than $1 - 5.54e^{-\lambda}$,
\begin{eqnarray*}
 \frac{|U_{1,n}(h,h_{\min})|}{n^2} & \leqslant &
 \frac{\theta\|K\|_{2}^{2}}{nh}
 \mathbb E(Y_{1}^{2})\\\
 & &
 +\mathfrak c_3\left(\frac{\|K\|_1\|K\|_{\infty}}{n^2h_{\min}}\mathfrak m(n)^2
 \left(\frac{\lambda^3}{\theta} +\lambda^2\right)\right.\\
 & &
 \left.
 +\frac{\|f\|_{\infty}\|K\|_{1}^{2}}{n\theta}\mathfrak m(n)^2
 (\lambda^2 +\lambda +\lambda^{1/2})
 \right)\\
 & \leqslant &
 \frac{\theta\|K\|_{2}^{2}}{nh}\mathbb E(Y_{1}^{2})\\
 & & +
 \frac{\mathfrak c_4}{\theta}\left(\frac{\|K\|_1\|K\|_{\infty}}{n^2h_{\min}}
 +\frac{\|f\|_{\infty}\|K\|_{1}^{2}}{n}
 \right)\mathfrak m(n)^2(1 +\lambda)^3.
\end{eqnarray*}
Then, with probability larger than $1 - 5.54|\mathcal H_n|e^{-\lambda}$,
\begin{eqnarray*}
 S_n(h_{\min})
 & \leqslant &
 \frac{\mathfrak c_4}{\theta}\left(\frac{\|K\|_1\|K\|_{\infty}}{n^2h_{\min}}
 +\frac{\|f\|_{\infty}\|K\|_{1}^{2}}{n}
 \right)\mathfrak m(n)^2(1 +\lambda)^3
\end{eqnarray*}
where
\begin{displaymath}
S_n(h_{\min}) :=
\sup_{h\in\mathcal H_n}\left\{
\frac{|U_{1,n}(h,h_{\min})|}{n^2} -\frac{\theta\|K\|_{2}^{2}}{nh}\mathbb E(Y_{1}^{2})\right\}.
\end{displaymath}
For every $s\in\mathbb R_+$, consider
\begin{displaymath}
\lambda(s) :=
-1 +
\left(\frac{s}{\mathfrak m(n,h_{\min},\theta)}\right)^{1/3},
\end{displaymath}
where
\begin{displaymath}
\mathfrak m(n,h_{\min},\theta) :=
\frac{\mathfrak c_4}{\theta}\left(\frac{\|K\|_1\|K\|_{\infty}}{n^2h_{\min}}
+\frac{\|f\|_{\infty}\|K\|_{1}^{2}}{n}
\right)\mathfrak m(n)^2.
\end{displaymath}
Then, for any $A > 0$,
\begin{eqnarray*}
 \mathbb E(S_n(h_{\min})) & \leqslant &
 A +
 \int_{A}^{\infty}\mathbb P(S_n(h_{\min})\geqslant s)ds\\
 & \leqslant &
 A + 5.54\mathfrak c_5|\mathcal H_n|
 \mathfrak m(n,h_{\min},\theta)\exp\left(
 -\frac{A^{1/3}}{2\mathfrak m(n,h_{\min},\theta)^{1/3}}
 \right)
\end{eqnarray*}
where
\begin{displaymath}
\mathfrak c_5 :=
\int_{0}^{\infty}e^{1 - s^{1/3}/2}ds.
\end{displaymath}
Since there exists a deterministic constant $\mathfrak c_6 > 0$, not depending on $n$ and $h_{\min}$ such that
\begin{displaymath}
\mathfrak m(n,h_{\min},\theta)\leqslant
\mathfrak c_6\frac{\log(n)^2}{n},
\end{displaymath}
by taking $A := 2^3\mathfrak c_6\log(n)^5/n$,
\begin{displaymath}
\mathbb E(S_n(h_{\min}))\leqslant
2^3\mathfrak c_6\frac{\log(n)^5}{n} +
5.54\mathfrak c_5
\mathfrak m(n,h_{\min},\theta)\frac{|\mathcal H_n|}{n}.
\end{displaymath}
Therefore, since $|\mathcal H_n|\leqslant n$, there exists a deterministic constant $\mathfrak c_7 > 0$, not depending on $n$ and $h_{\min}$, such that
\begin{displaymath}
\mathbb E\left(\sup_{h\in\mathcal H_n}\left\{
\frac{|U_{1,n}(h,h_{\min})|}{n^2} -\frac{\theta\|K\|_{2}^{2}}{nh}\mathbb E(Y_{1}^{2})\right\}\right)\leqslant
\frac{\mathfrak c_7}{\theta}\cdot\frac{\log(n)^5}{n}.
\end{displaymath}
On the other hand,
\begin{displaymath}
U_n(h,h_{\min}) =
\sum_{i = 1}^{4}U_{i,n}(h,h_{\min})
\end{displaymath}
where, for $i = 2,3,4$,
\begin{displaymath}
U_{i,n}(h,h_{\min}) :=
\sum_{k\not= l}g_{n,h,h_{\min}}^{i}(Z_k,Z_l)
\end{displaymath}
with
\begin{eqnarray*}
 g_{n,h,h_{\min}}^{2}(z,z') & := &
 \langle z_2\mathbf 1_{|z_2|\leqslant\mathfrak m(n)}K_h(z_1 -\cdot),
 z_2'\mathbf 1_{|z_2'| >\mathfrak m(n)}K_{h_{\min}}(z_1' -\cdot)\rangle_2,\\
 g_{n,h,h_{\min}}^{3}(z,z') & := &
 \langle z_2\mathbf 1_{|z_2| >\mathfrak m(n)}K_h(z_1 -\cdot),
 z_2'\mathbf 1_{|z_2'|\leqslant\mathfrak m(n)}K_{h_{\min}}(z_1' -\cdot)\rangle_2
 \textrm{ and}\\
 g_{n,h,h_{\min}}^{4}(z,z') & := &
 \langle z_2\mathbf 1_{|z_2| >\mathfrak m(n)}K_h(z_1 -\cdot),
 z_2'\mathbf 1_{|z_2'| >\mathfrak m(n)}K_{h_{\min}}(z_1' -\cdot)\rangle_2
\end{eqnarray*}
for every $z,z'\in\mathbb R^2$. Consider $k,l\in\{1,\dots,n\}$ such that $k\not= l$. By Markov's inequality,
\begin{eqnarray*}
 & &
 \mathbb E\left(\sup_{h\in\mathcal H_n}|g_{n,h,h_{\min}}^{2}(Z_k,Z_l)|\right)\\
 & &
 \hspace{0.5cm}\leqslant
 \mathfrak m(n)
 \sum_{h\in\mathcal H_n}\int_{-\infty}^{\infty}\mathbb E(|K_h(X_k - u)|)
 \mathbb E(|Y_l|\mathbf 1_{|Y_l| >\mathfrak m(n)}|K_{h_{\min}}(X_l - u)|)du\\
 & &
 \hspace{0.5cm}\leqslant
 \mathfrak m(n)|\mathcal H_n|\cdot\|f\|_{\infty}\|K\|_{1}^{2}
 \mathbb E(Y_{1}^{2})^{1/2}
 \mathbb P(\exp(\alpha|Y_1|) > n^4)^{1/2}\\
 & &
 \hspace{0.5cm}\leqslant
 \|f\|_{\infty}\|K\|_{1}^{2}\mathbb E(Y_{1}^{2})^{1/2}
 \mathbb E(\exp(\alpha|Y_1|))^{1/2}\frac{\mathfrak m(n)}{n^2}|\mathcal H_n|.
\end{eqnarray*}
Then, there exists a deterministic constant $\mathfrak c_8 > 0$, not depending on $n$ and $h_{\min}$, such that
\begin{displaymath}
\mathbb E\left(\sup_{h\in\mathcal H_n}\frac{|U_{2,n}(h,h_{\min})|}{n^2}\right)
\leqslant
\mathfrak c_8\frac{\log(n)}{n}.
\end{displaymath}
The same ideas give that there exists a deterministic constant $\mathfrak c_9 > 0$, not depending on $n$ and $h_{\min}$, such that
\begin{displaymath}
\mathbb E\left(\sup_{h\in\mathcal H_n}\frac{|U_{3,n}(h,h_{\min})|}{n^2}\right)
\leqslant
\mathfrak c_9\frac{\log(n)}{n}.
\end{displaymath}
For $i = 4$, by Markov's inequality,
\begin{eqnarray*}
 & &
 \mathbb E\left(\sup_{h\in\mathcal H_n}
 |g_{n,h,h_{\min}}^{4}(Z_k,Z_l)|\right)\\
 & &
 \hspace{0.5cm}
 \leqslant
 \sum_{h\in\mathcal H_n}
 \int_{-\infty}^{\infty}\mathbb E(|Y_k|\mathbf 1_{|Y_k| >\mathfrak m(n)}|K_h(X_k - u)|)\\
 & &
 \hspace{3cm}
 \times\mathbb E(|Y_l|\mathbf 1_{|Y_l| >\mathfrak m(n)}|K_{h_{\min}}(X_l - u)|)du\\
 & &
 \hspace{0.5cm}\leqslant
 \frac{\|K\|_{\infty}}{h_{\min}}\mathbb E(|Y_l|\mathbf 1_{|Y_l| >\mathfrak m(n)})\\
 & &
 \hspace{3cm}
 \times
 \sum_{h\in\mathcal H_n}
 \int_{-\infty}^{\infty}\mathbb E(|Y_k|\mathbf 1_{|Y_k| >\mathfrak m(n)}|K_h(X_k - u)|)du\\
 & &
 \hspace{0.5cm}\leqslant
 \frac{\|K\|_{\infty}\|K\|_1}{h_{\min}}|\mathcal H_n|\cdot\mathbb E(Y_{1}^{2})\mathbb P(|Y_1| >\mathfrak m(n))\\
 & &
 \hspace{0.5cm}\leqslant
 \|K\|_{\infty}\|K\|_1
 \mathbb E(Y_{1}^{2})
 \mathbb E(\exp(\alpha|Y_1|))\frac{1}{n^4h_{\min}}|\mathcal H_n|.
\end{eqnarray*}
Then, there exists a deterministic constant $\mathfrak c_{10} > 0$, not depending on $n$ and $h_{\min}$, such that
\begin{displaymath}
\mathbb E\left(\sup_{h\in\mathcal H_n}\frac{|U_{4,n}(h,h_{\min})|}{n^2}\right)
\leqslant
\mathfrak c_{10}\frac{\log(n)}{n^3h_{\min}}.
\end{displaymath}
Therefore,
\begin{displaymath}
\mathbb E\left(\sup_{h\in\mathcal H_n}\left\{
\frac{|U_n(h,h_{\min})|}{n^2} -\frac{\theta\|K\|_{2}^{2}}{nh}\mathbb E(Y_{1}^{2})\right\}\right)
\leqslant
\frac{\mathfrak c_U}{\theta}\cdot\frac{\log(n)^5}{n}.
\end{displaymath}
%


%
\subsubsection{Proof of Lemma \ref{bound_V}}
Consider $\mathfrak m(n) := 4\log(n)/\alpha$. For any $\eta,\eta'\in\mathcal H_n$,
\begin{displaymath}
V_n(\eta,\eta') = V_{1,n}(\eta,\eta') + V_{2,n}(\eta,\eta')
\end{displaymath}
where
\begin{displaymath}
V_{i,n}(\eta,\eta') :=
\frac{1}{n}\sum_{k = 1}^{n}(g_{\eta,\eta'}^{i}(X_k,Y_k)
-\mathbb E(g_{\eta,\eta'}^{i}(X_k,Y_k)))
\textrm{ $;$ }
i = 1,2
\end{displaymath}
with, for every $x,y\in\mathbb R$,
\begin{displaymath}
g_{\eta,\eta'}^{1}(x,y) :=
\langle yK_{\eta}(x -\cdot),
(bf)_{\eta'} - bf\rangle_2\mathbf 1_{|y|\leqslant\mathfrak m(n)}
\end{displaymath}
and
\begin{displaymath}
g_{\eta,\eta'}^{2}(x,y) :=
\langle yK_{\eta}(x -\cdot),
(bf)_{\eta'} - bf\rangle_2\mathbf 1_{|y| >\mathfrak m(n)}.
\end{displaymath}
In order to apply Bernstein's inequality to $g_{\eta,\eta'}^{1}(X_k,Y_k)$, $k = 1,\dots,n$, let us find suitable controls of
\begin{displaymath}
\mathfrak c_{\eta,\eta'} :=
\frac{\|g_{\eta,\eta'}^{1}\|_{\infty}}{3}
\textrm{ and }
\mathfrak v_{\eta,\eta'} :=
\mathbb E(g_{\eta,\eta'}^{1}(X_1,Y_1)^2).
\end{displaymath}
On the one hand, since $\|K\|_1\geqslant 1$ and $bf$ is bounded,
\begin{eqnarray*}
 \mathfrak c_{\eta,\eta'} & = &
 \frac{1}{3}\sup_{x,y\in\mathbb R}
 |\langle yK_{\eta}(x -\cdot),(bf)_{\eta'} - bf\rangle_2
 \mathbf 1_{|y|\leqslant\mathfrak m(n)}|\\
 & \leqslant &
 \frac{\mathfrak m(n)}{3}\|(bf)_{\eta'} - bf\|_{\infty}\sup_{x\in\mathbb R}
 \|K_{\eta}(x -\cdot)\|_1\\
 & \leqslant &
 \frac{\mathfrak m(n)}{3}
 \|K\|_1(\|K\|_1 + 1)\|bf\|_{\infty}
 \leqslant
 \frac{2}{3}\mathfrak m(n)\|K\|_{1}^{2}\|bf\|_{\infty}.
\end{eqnarray*}
On the other hand,
\begin{eqnarray*}
 \mathfrak v_{\eta,\eta'} & = &
 \mathbb E(\langle Y_1K_{\eta}(X_1 -\cdot),(bf)_{\eta'} - bf\rangle_{2}^{2}
 \mathbf 1_{|Y_1|\leqslant\mathfrak m(n)})\\
 & = &
 \mathbb E\left(Y_{1}^{2}\mathbf 1_{|Y_1|\leqslant\mathfrak m(n)}
 \left|\int_{-\infty}^{\infty}K_{\eta}(X_1 - u)((bf)_{\eta'}(u) - (bf)(u))du\right|^2\right)\\
 & \leqslant &
 \mathfrak m(n)^2\|f\|_{\infty}\|K\|_{1}^{2}\|(bf)_{\eta'} - bf\|_{2}^{2}.
\end{eqnarray*}
So, by Bernstein's inequality, there exists a universal constant $\mathfrak c_1 > 0$ such that with probability larger than $1 - 2e^{-\lambda}$,
\begin{eqnarray*}
 |V_{1,n}(\eta,\eta')| & \leqslant &
 \sqrt{\frac{2\lambda}{n}\mathfrak v_{\eta,\eta'}} +
 \frac{\lambda}{n}\mathfrak c_{\eta,\eta'}\\
 & \leqslant &
 \theta\|(bf)_{\eta'} - bf\|_{2}^{2} +
 \mathfrak c_1\frac{\mathfrak m(n)^2}{\theta n}
 \|K\|_{1}^{2}(\|f\|_{\infty} +\|bf\|_{\infty})\lambda.
\end{eqnarray*}
Then, with probability larger than $1 - 2|\mathcal H_n|^2e^{-\lambda}$,
\begin{displaymath}
S_n\leqslant
\mathfrak c_1\frac{\mathfrak m(n)^2}{\theta n}
\|K\|_{1}^{2}(\|f\|_{\infty} +\|bf\|_{\infty})\lambda
\end{displaymath}
where
\begin{displaymath}
S_n :=
\sup_{\eta,\eta'\in\mathcal H_n}\{
|V_{1,n}(\eta,\eta')| -\theta\|(bf)_{\eta'} - bf\|_{2}^{2}\}.
\end{displaymath}
For every $s\in\mathbb R_+$, consider
\begin{displaymath}
\lambda(s) :=
\frac{s}{\mathfrak m(n,\theta)}
\quad{\rm with}\quad
\mathfrak m(n,\theta) :=
\mathfrak c_1\frac{\mathfrak m(n)^2}{\theta n}
\|K\|_{1}^{2}(\|f\|_{\infty} +\|bf\|_{\infty}).
\end{displaymath}
Then, for any $A > 0$,
\begin{eqnarray*}
 \mathbb E(S_n) & \leqslant &
 A +
 \int_{A}^{\infty}\mathbb P(S_n\geqslant s)ds\\
 & \leqslant &
 A + 2\mathfrak c_2|\mathcal H_n|^2
 \mathfrak m(n,\theta)\exp\left(
 -\frac{A}{2\mathfrak m(n,\theta)}
 \right)
\end{eqnarray*}
where $\int_{0}^{\infty}e^{-s/2}ds = 2$. Since there exists a deterministic constant $\mathfrak c_3 > 0$, not depending on $n$ and $h_{\min}$ such that
\begin{displaymath}
\mathfrak m(n,\theta)\leqslant
\mathfrak c_3\frac{\log(n)^2}{n},
\end{displaymath}
by taking $A := 4\mathfrak c_3\log(n)^3/n$,
\begin{displaymath}
\mathbb E(S_n)\leqslant
4\mathfrak c_3\frac{\log(n)^3}{n} +
2\mathfrak c_2
\mathfrak m(n,\theta)\frac{|\mathcal H_n|}{n^2}.
\end{displaymath}
Therefore, since $|\mathcal H_n|\leqslant n$, there exists a deterministic constant $\mathfrak c_4 > 0$, not depending on $n$ and $h_{\min}$, such that
\begin{displaymath}
\mathbb E\left(\sup_{\eta,\eta'\in\mathcal H_n}\{
|V_{1,n}(\eta,\eta')| -\theta\|(bf)_{\eta'} - bf\|_{2}^{2}\}\right)
\leqslant
\frac{\mathfrak c_4}{\theta}\cdot\frac{\log(n)^3}{n}.
\end{displaymath}
Now, let us find a suitable control of
\begin{displaymath}
\mathfrak v_{2,n} :=
\mathbb E\left(
\sup_{\eta,\eta'\in\mathcal H_n}
|V_{2,n}(\eta,\eta')|\right).
\end{displaymath}
By Markov's inequality,
\begin{eqnarray*}
 \mathfrak v_{2,n} & \leqslant &
 2\mathbb E\left(\sup_{\eta\in\mathcal H_n}|Z_{2,1}(\eta,\eta')|\right)\\
 & \leqslant &
 2\mathbb E(Y_{1}^{2}\mathbf 1_{|Y_1| >\mathfrak m(n)})^{1/2}\\
 & &
 \times
 \mathbb E\left(\sup_{\eta,\eta'\in\mathcal H_n}\left|
 \int_{-\infty}^{\infty}K_{\eta}(X_1 - u)((bf)_{\eta'}(u) - (bf)(u))du\right|^2\right)^{1/2}\\
 & \leqslant &
 2\mathbb E(Y_{1}^{4})^{1/4}\mathbb P(\exp(\alpha|Y_1|) > n^4)^{1/4}\|K\|_1
 \sup_{\eta'\in\mathcal H_n}\|(bf)_{\eta'} - bf\|_{\infty}\\
 & \leqslant &
 2\mathbb E(Y_{1}^{4})^{1/4}\mathbb E(\exp(\alpha|Y_1|))^{1/4}\|K\|_{1}^{2}\|bf\|_{\infty}
 \frac{1}{n}.
\end{eqnarray*}
Therefore,
\begin{displaymath}
\mathbb E\left(\sup_{\eta,\eta'\in\mathcal H_n}\{
|V_n(\eta,\eta')| -\theta\|(bf)_{\eta'} - bf\|_{2}^{2}\}\right)
\leqslant
\mathfrak c_V\frac{\log(n)^3}{\theta n}.
\end{displaymath}
%


%
\subsubsection{Proof of Lemma \ref{Lerasle_type_inequality}}
First of all,
\begin{displaymath}
\|(bf)_h - bf\|_{2}^{2} =
\|\widehat{bf}_{n,h} - bf\|_{2}^{2} -
\|\widehat{bf}_{n,h} - (bf)_h\|_{2}^{2} - 2V_n(h,h).
\end{displaymath}
Then, for any $\theta\in (0,1/2)$,
\begin{eqnarray}
 & &
 (1 - 2\theta)\left(\|(bf)_h - bf\|_{2}^{2} +
 \frac{\mathfrak c_{K,Y}}{nh}\right) -\|\widehat{bf}_{n,h} - bf\|_{2}^{2}
 \nonumber\\
 & &
 \label{Lerasle_type_inequality_1}
 \hspace{2cm}\leqslant
 2(|V_n(h,h)| -\theta\|(bf)_h - bf\|_{2}^{2})
 +\Lambda_n(h) - 2\theta\frac{\mathfrak c_{K,Y}}{nh}
\end{eqnarray}
where
\begin{eqnarray*}
 \Lambda_n(h) & := &
 \left|\|\widehat{bf}_{n,h} - (bf)_h\|_{2}^{2} -\frac{\mathfrak c_{K,Y}}{nh}\right|\\
 & = &
 \left|\frac{U_n(h,h)}{n^2} +\frac{W_n(h)}{n} -\frac{1}{n}\|(bf)_h\|_{2}^{2}\right|
\end{eqnarray*}
with
\begin{displaymath}
W_n(h) :=
\frac{1}{n}\sum_{k = 1}^{n}(Z_k(h) -\mathbb E(Z_k(h)))
\end{displaymath}
and
\begin{displaymath}
Z_k(h) :=\|Y_kK_h(X_k -\cdot) - (bf)_h\|_{2}^{2}
\textrm{ $;$ }
\forall k\in\{1,\dots,n\},
\end{displaymath}
because
\begin{eqnarray*}
 \mathbb E(Z_1(h)) & = &
 \sigma^2\int_{-\infty}^{\infty}\mathbb E(K_h(X_1 - u)^2)du +
 \int_{-\infty}^{\infty}\mathbb E(b(X_1)^2K_h(X_1 - u)^2)du\\
 & & -
 2\int_{-\infty}^{\infty}\mathbb E(b(X_1)K_h(X_1 - u))(bf)_h(u)du +
 \int_{-\infty}^{\infty}(bf)_h(u)^2du\\
 & = &
 \frac{\|K\|_{2}^{2}}{h}(\sigma^2 +\mathbb E(b(X_1)^2)) -\|(bf)_h\|_{2}^{2} =
 \frac{\mathfrak c_{K,Y}}{h} -\|(bf)_h\|_{2}^{2}.
\end{eqnarray*}
Consider $\mathfrak m(n) := 2\log(n)/\alpha$ and note that $W_n(h) = W_{1,n}(h) + W_{2,n}(h)$, where
\begin{displaymath}
W_{i,n}(h) :=
\frac{1}{n}\sum_{k = 1}^{n}(g_{h}^{i}(X_k,Y_k) -\mathbb E(g_{h}^{i}(X_k,Y_k)))
\textrm{ $;$ }i = 1,2
\end{displaymath}
with, for every $x,y\in\mathbb R$,
\begin{displaymath}
g_{h}^{1}(x,y) :=
\|yK_h(x -\cdot) - (bf)_h\|_{2}^{2}\mathbf 1_{|y|\leqslant\mathfrak m(n)}
\end{displaymath}
and
\begin{displaymath}
g_{h}^{2}(x,y) :=
\|yK_h(x -\cdot) - (bf)_h\|_{2}^{2}\mathbf 1_{|y| >\mathfrak m(n)}.
\end{displaymath}
Note also that
\begin{displaymath}
\|(bf)_h\|_2\leqslant\|K_h\|_2\int_{-\infty}^{\infty}|b(x)|f(x)dx
\leqslant
\frac{\|K\|_2}{h^{1/2}}\mathbb E(|b(X_1)|)
\leqslant
\left(\frac{\mathfrak c_{K,Y}}{h}\right)^{1/2}
\end{displaymath}
and
\begin{displaymath}
\|(bf)_h\|_2\leqslant\|K_h\|_1\left(\int_{-\infty}^{\infty}b(x)^2f(x)^2dx\right)^{1/2}
\leqslant
\|K\|_1\|f\|_{\infty}^{1/2}\mathbb E(b(X_1)^2)^{1/2}.
\end{displaymath}
In order to apply Bernstein's inequality to $g_{h}^{1}(X_k,Y_k)$, $k = 1,\dots,n$, let us find suitable controls of
\begin{displaymath}
\mathfrak c_h :=
\frac{\|g_{h}^{1}\|_{\infty}}{3}
\textrm{ and }
\mathfrak v_h :=
\mathbb E(g_{h}^{1}(X_1,Y_1)^2).
\end{displaymath}
On the one hand,
\begin{eqnarray*}
 \mathfrak c_h & = &
 \frac{1}{3}\sup_{x,y\in\mathbb R}
 \|yK_h(x -\cdot) - (bf)_h\|_{2}^{2}
 \mathbf 1_{|y|\leqslant\mathfrak m(n)}\\
 & \leqslant &
 \frac{2}{3}\left(\mathfrak m(n)^2\frac{\|K\|_{2}^{2}}{h} +\frac{\mathfrak c_{K,Y}}{h}\right).
\end{eqnarray*}
On the other hand,
\begin{eqnarray*}
 \mathfrak v_h
 & \leqslant &
 2\mathbb E(Z_1(h)(\|Y_1K_h(X_1 -\cdot)\|_{2}^{2}\mathbf 1_{|Y_1|\leqslant\mathfrak m(n)} +\|(bf)_h\|_{2}^{2}))\\
 & \leqslant &
 \frac{2}{h}\mathbb E(Z_1(h))(\|K\|_{2}^{2}\mathfrak m(n)^2 +\mathfrak c_{K,Y})
 \leqslant
 2(\|K\|_{2}^{2} +\mathfrak c_{K,Y})
 \frac{\mathfrak c_{K,Y}}{hh_{\min}}\mathfrak m(n)^2.
\end{eqnarray*}
So, by Bernstein's inequality, there exists a universal constant $\mathfrak c_1 > 0$ such that with probability larger than $1 - 2e^{-\lambda}$,
\begin{eqnarray*}
 |W_{1,n}(h)| & \leqslant &
 \sqrt{\frac{2\lambda}{n}\mathfrak v_h} +
 \frac{\lambda}{n}\mathfrak c_h\\
 & \leqslant &
 \theta\frac{\mathfrak c_{K,Y}}{h} +
 \mathfrak c_1\frac{\mathfrak m(n)^2}{\theta nh_{\min}}(\|K\|_{2}^{2} +\mathfrak c_{K,Y})\lambda.
\end{eqnarray*}
Then, with probability larger than $1 - 2|\mathcal H_n|e^{-\lambda}$,
\begin{displaymath}
S_n(h_{\min})\leqslant
\mathfrak c_1\frac{\mathfrak m(n)^2}{\theta n^2h_{\min}}(\|K\|_{2}^{2} +\mathfrak c_{K,Y})\lambda
\end{displaymath}
where
\begin{displaymath}
S_n(h_{\min}) :=
\sup_{h\in\mathcal H_n}\left\{
\frac{|W_{1,n}(h)|}{n} -\theta\frac{\mathfrak c_{K,Y}}{nh}\right\}.
\end{displaymath}
For every $s\in\mathbb R_+$, consider
\begin{displaymath}
\lambda(s) :=
\frac{s}{\mathfrak m(n,h_{\min},\theta)}
\quad{\rm with}\quad
\mathfrak m(n,h_{\min},\theta) :=
\mathfrak c_1\frac{\mathfrak m(n)^2}{\theta n^2h_{\min}}(\|K\|_{2}^{2} +\mathfrak c_{K,Y}).
\end{displaymath}
Then, for any $A > 0$,
\begin{eqnarray*}
 \mathbb E(S_n(h_{\min})) & \leqslant &
 A +
 \int_{A}^{\infty}\mathbb P(S_n(h_{\min})\geqslant s)ds\\
 & \leqslant &
 A + 2\mathfrak c_2|\mathcal H_n|
 \mathfrak m(n,h_{\min},\theta)\exp\left(
 -\frac{A}{2\mathfrak m(n,h_{\min},\theta)}
 \right)
\end{eqnarray*}
where $\mathfrak c_2 :=\int_{0}^{\infty}e^{-s/2}ds = 2$. Since there exists a deterministic constant $\mathfrak c_3 > 0$, not depending on $n$ and $h_{\min}$, such that
\begin{displaymath}
\mathfrak m(n,\theta)\leqslant
\mathfrak c_3\frac{\log(n)^2}{n},
\end{displaymath}
by taking $A := 2\mathfrak c_3\log(n)^3/n$,
\begin{displaymath}
\mathbb E(S_n(h_{\min}))\leqslant
2\mathfrak c_3\frac{\log(n)^3}{n} +
2\mathfrak c_2
\mathfrak m(n,h_{\min},\theta)\frac{|\mathcal H_n|}{n}.
\end{displaymath}
Therefore, since $|\mathcal H_n|\leqslant n$, there exists a deterministic constant $\mathfrak c_4 > 0$, not depending on $n$ and $h_{\min}$, such that
\begin{displaymath}
\mathbb E\left(\sup_{h\in\mathcal H_n}\left\{
\frac{|W_{1,n}(h)|}{n} -\theta\frac{\mathfrak c_{K,Y}}{nh}\right\}\right)
\leqslant
\frac{\mathfrak c_4}{\theta}\cdot\frac{\log(n)^3}{n}.
\end{displaymath}
Now, by Markov's inequality,
\begin{eqnarray*}
 \mathbb E\left(\sup_{h\in\mathcal H_n}\frac{|W_{2,n}(h)|}{n}\right) & \leqslant &
 \frac{2}{n}\mathbb E\left(\sup_{h\in\mathcal H_n}|Z_1(h)|\mathbf 1_{|Y_1| >\mathfrak m(n)}\right)\\
 & \leqslant &
 \frac{4}{n}\mathbb E\left(\sup_{h\in\mathcal H_n}(\|Y_1K_h(X_1 -\cdot)\|_{2}^{2} +\|(bf)_h\|_{2}^{2})\mathbf 1_{|Y_1| >\mathfrak m(n)}\right)\\
 & \leqslant &
 \frac{4}{nh_{\min}}(\|K\|_{2}^{2}\mathbb E(Y_{1}^{4})^{1/2} +
 \mathfrak c_{K,Y})\mathbb P(|Y_1| >\mathfrak m(n))^{1/2}\\
 & \leqslant &
 4(\|K\|_{2}^{2}\mathbb E(Y_{1}^{4})^{1/2} +
 \mathfrak c_{K,Y})
 \mathbb E(\exp(\alpha|Y_1|))^{1/2}
 \frac{1}{\theta n^2h_{\min}}.
\end{eqnarray*}
Then, there exists a deterministic constant $\mathfrak c_5 > 0$, not depending on $n$ and $h_{\min}$, such that
\begin{displaymath}
\mathbb E\left(\sup_{h\in\mathcal H_n}\left\{
\frac{|W_n(h)|}{n} -\theta\frac{\mathfrak c_{K,Y}}{nh}\right\}\right)
\leqslant
\frac{\mathfrak c_5}{\theta}\cdot\frac{\log(n)^3}{n}.
\end{displaymath}
Therefore, by Lemma \ref{bound_U}, there exists a deterministic constant $\mathfrak c_6 > 0$, not depending on $n$ and $h_{\min}$, such that
\begin{eqnarray*}
 \mathbb E\left(\sup_{h\in\mathcal H_n}\left\{
 \Lambda_n(h) - 2\theta\frac{\mathfrak c_{K,Y}}{nh}\right\}\right) & \leqslant &
 \frac{\mathfrak c_U}{\theta}\cdot\frac{\log(n)^5}{n} +\frac{\mathfrak c_5}{\theta}\cdot\frac{\log(n)^3}{n}\\
 & & +
 \frac{1}{n}\|K\|_{1}^{2}\|f\|_{\infty}\mathbb E(b(X_1)^2)\\
 & \leqslant &
 \frac{\mathfrak c_6}{\theta}\cdot\frac{\log(n)^5}{n}.
\end{eqnarray*}
Moreover, by Lemma \ref{bound_V},
\begin{displaymath}
\mathbb E\left(\sup_{h\in\mathcal H_n}
\{|V_n(h,h)| -\theta\|(bf)_h - bf\|_{2}^{2}\}\right)
\leqslant\frac{\mathfrak c_V}{\theta}\cdot\frac{\log(n)^3}{n}.
\end{displaymath}
In conclusion, by Inequality (\ref{Lerasle_type_inequality_1}),
\begin{eqnarray*}
 & &
 \mathbb E\left(
 \sup_{h\in\mathcal H_n}\left\{
 \|(bf)_h - bf\|_{2}^{2} +\frac{\mathfrak c_{K,Y}}{nh} -\frac{1}{1 - 2\theta}
 \|\widehat{bf}_{n,h} - bf\|_{2}^{2}\right\}\right)\\
 & &
 \hspace{7cm}
 \leqslant
 \frac{\mathfrak c_L}{\theta(1 - 2\theta)}
 \cdot\frac{\log(n)^5}{n}.
\end{eqnarray*}
%


%

%


%
\appendix
\section{Additional simulation results}\label{Append}
\begin{table}[h!]
\begin{tabular}{c|cccc}
 $n$ & $b_1$ & $b_2$ & $b_3$ & $b_4$ \\ \hline
 $250$ &  0.32 &   0.29   & 0.14  &  0.32 \\
 & {\small (0.09)} &   {\small (0.08)}   & {\small (0.03)}  & {\small (0.10)} \\
 $500$ &  0.29 &   0.25  &  0.12 &   0.27 \\
 & {\small (0.07)}  &  {\small (0.06)} &   {\small (0.02)}  &  {\small (0.08)} \\
 $1000$ &  0.25 &  0.22  &  0.10 &   0.22 \\
 & {\small (0.05)} &  {\small (0.04)} &  {\small (0.01)}  & {\small (0.05)} \\
\end{tabular}
\caption{Mean of selected bandwidth  (with std in parenthesis below) with the CV method for NW-single bandwidth estimator of $b$, $\sigma = 0.7$, $X\sim\mathcal N(0,1)$, 200 repetitions. }\label{tab_b_bandwithNWsig07}
\end{table}
\begin{table}[h!]
\hspace{-0cm}\begin{tabular}{c|ccc|ccc}
 & \multicolumn{3}{c}{$b_1f$} & \multicolumn{3}{c}{$b_2f$}\\
 $n$ & PCO & CV & Or & PCO & CV & Or\\ \hline
 250 & 0.39 & 0.46 & 0.20 & 0.51 & 0.57 & 0.25\\
 & {\small (0.35)} & {\small (0.56)} & {\small (0.19)} & {\small (0.44)} & {\small (0.81)} &  {\small (0.23)}\\
 $500$ & 0.19 & 0.27 & 0.10 & 0.27 & 0.32 & 0.13\\
 & {\small (0.16)} & {\small (0.37)} & {\small (0.09)} & {\small (0.21)} & {\small (0.38)} & {\small (0.13)}\\
 $1000$ & 0.11 & 0.18 & 0.06 & 0.15 & 0.23 & 0.07\\
 & {\small (0.09)} & {\small (0.30)} & {\small (0.05)} & {\small (0.12)} & {\small (0.44)} & {\small (0.06)}\\
 \multicolumn{6}{c}{}\\
\end{tabular}
\begin{tabular}{c|ccc|ccc}
 & \multicolumn{3}{c}{$b_3f$} & \multicolumn{3}{c}{$b_4f$}\\
 $n$ & PCO & CV & Or & PCO & CV & Or\\ \hline
 250 & 0.61 & 0.71 & 0.42 & 0.20 & 0.22 & 0.10\\
 & {\small (0.37)} & {\small (0.70)} & {\small (0.27)} & {\small (0.16)} & {\small (0.26)} & {\small (0.09)}\\
 $500$ & 0.31 & 0.34 & 0.22 & 0.10 & 0.13 & 0.05\\
 & {\small (0.18)} & {\small (0.30)} & {\small (0.14)} & {\small (0.08)} & {\small (0.17)} & {\small (0.04)}\\
 $1000$ & 0.16 & 0.24 & 0.11 & 0.05 & 0.08 & 0.03\\
 & {\small (0.11)} & {\small (0.41)} & {\small (0.07)} & {\small (0.04)} & {\small (0.10)} & {\small (0.02)}\\
 \multicolumn{6}{c}{}\\
\end{tabular}
\caption{100*MISE (with 100*std in parenthesis below) for the estimation of $bf$, 200 repetitions, $X\sim\gamma(3,2)/5$ and $\sigma = 0.1$. Same columns as in Table \ref{tab_bf_XGauss}.}\label{tab_bf_XGamma}
\end{table}
\begin{table}[h!]
\hspace{-0cm}\begin{tabular}{c|ccc|ccc}
 & \multicolumn{3}{c}{$b_1f$} & \multicolumn{3}{c}{$b_2f$}\\
 $n$ & PCO & CV & Or & PCO & CV & Or\\ \hline
 250 & 0.91 & 0.85 & 0.49 & 0.83 & 0.93 & 0.47\\
 & {\small (0.84)} & {\small (0.74)} & {\small (0.40)} & {\small (0.74)} & {\small (1.24)} &  {\small (0.36)}\\
 $500$ & 0.43 & 0.44 & 0.23 & 0.47 & 0.48 & 0.25\\
 & {\small (0.30)} & {\small (0.39)} & {\small (0.17)} & {\small (0.32)} & {\small (0.53)} & {\small (0.19)}\\
 $1000$ & 0.22 & 0.23 & 0.13 & 0.24 & 0.24 & 0.13\\
 & {\small (0.14)} & {\small (0.23)} & {\small (0.07)} & {\small (0.16)} & {\small (0.22)} & {\small (0.09)}\\
 \multicolumn{6}{c}{}\\
\end{tabular}
\begin{tabular}{c|ccc|ccc}
 & \multicolumn{3}{c}{$b_3f$} & \multicolumn{3}{c}{$b_4f$}\\
 $n$ & PCO & CV & Or & PCO & CV & Or\\ \hline
 250 & 1.21 & 1.18 & 0.80 & 0.66 & 0.60 & 0.37\\
 & {\small (0.87)} & {\small (0.86)} & {\small (0.54)} & {\small (0.55)} & {\small (0.49)} &  {\small (0.29)}\\
 $500$ & 0.56 & 0.55 & 0.39 & 0.34 & 0.32 & 0.18\\
 & {\small (0.30)} & {\small (0.42)} & {\small (0.22)} & {\small (0.25)} & {\small (0.29)} & {\small (0.14)}\\
 $1000$ & 0.28 & 0.29 & 0.20 & 0.17 & 0.17 & 0.07\\
 & {\small (0.16)} & {\small (0.26)} & {\small (0.12)} & {\small (0.11)} & {\small (0.17)} & {\small (0.07)}\\
 \multicolumn{6}{c}{}\\
\end{tabular}
\caption{100*MISE (with 100*std in parenthesis below) for the estimation of $bf$, 200 repetitions, $X\sim\gamma(3,2)/5$ and $\sigma = 0.7$. Same columns as in Table \ref{tab_bf_XGauss}.}\label{tab_bf_XGamma07}
\end{table}
\begin{table}[h!]
\hspace{-0cm}\begin{tabular}{c|ccc|ccc}
 & \multicolumn{3}{c}{$b_1$} & \multicolumn{3}{c}{$b_2$}\\
 $n$ & CV & PCO & Or & CV & PCO & Or\\ \hline
 $250$ & 0.22 & 0.66 & 0.39 & 0.56 & 8.75 & 1.74\\
 & {\small (0.15)} & {\small (0.42)} & {\small (0.34)} & {\small (0.78)} & {\small (39.2)} &  {\small (2.50)}\\ 
 $500$ & 0.11 & 0.30 & 0.17 & 0.43 & 1.35 & 0.67\\
 & {\small (0.06)} & {\small (0.23)} & {\small (0.14)} & {\small (0.16)} & {\small (2.29)} & {\small (0.91)}\\
 $1000$ & 0.07 & 0.15 & 0.09 & 0.12 & 0.38 & 0.28\\
 & {\small (0.05)} & {\small (0.09)} & {\small (0.09)} & {\small (0.08)} & {\small (0.41)} & {\small (0.30)}\\
 \multicolumn{6}{c}{} \\
\end{tabular}
\begin{tabular}{c|ccc|ccc}
 & \multicolumn{3}{c}{$b_3$} & \multicolumn{3}{c}{$b_4$} \\
 $n$ & CV & PCO & Or & CV & PCO & Or\\ \hline
 $250$ & 1.07 & 7.50 & 3.80 & 0.24 & 0.59 & 0.38\\
 & {\small (1.42)} & {\small (20.6)} & {\small (4.17)} & {\small (0.16)} & {\small (0.28)} & {\small (0.28)}\\
 $500$ & 0.42 & 2.38 & 1.72 & 0.13 & 0.33 & 0.19\\
 & {\small (0.26)} & {\small (1.67)} & {\small (1.50)} & {\small (0.07)} & {\small (0.15)} & {\small (0.13)}\\
 $1000$ & 0.21 & 1.05 & 0.74 & 0.08 & 0.19 & 0.11\\
 & {\small (0.11)} & {\small (0.67)} & {\small (0.56)} & {\small (0.05)} & {\small (0.08)} & {\small (0.08)}\\
\multicolumn{6}{c}{}\\
\end{tabular}
\caption{100*MISE (with 100*std in parenthesis below) for the estimation of $b_i$, $i=1, \dots, 4$, 200 repetitions, $X\sim\gamma(3,2)/5$, $\sigma = 0.1$. CV and PCO are the two competing methods. Column "Or" gives the average of ISE for the ratio of the two best estimators of $bf$ and $f$ in the collection.}\label{tab_b_XGamma}
\end{table}
\begin{table}[h!]
\hspace{-0cm}\begin{tabular}{c|ccc|ccc}
 & \multicolumn{3}{c}{$b_1$} & \multicolumn{3}{c}{$b_2$}\\
 $n$ & CV & PCO & Or & CV & PCO & Or\\ \hline
 $250$ & 4.86 & 7.99 & 6.42 & 6.77 & 16.0 & 7.08\\
 & {\small (4.72)} & {\small (9.41)} & {\small (7.65)} & {\small (6.04)} & {\small (43.7)} &  {\small (8.63)}\\
 $500$ & 2.58 & 2.87 & 3.12 & 3.74 & 3.85 & 3.37\\
 & {\small (2.18)} & {\small (2.19)} & {\small (2.95)} & {\small (4.06)} & {\small (4.06)} & {\small (2.88)}\\
 $1000$ & 1.51 & 1.35 & 1.47 & 1.94 & 1.62 & 1.67\\
 & {\small (1.42)} & {\small (1.16)} & {\small (1.19)} & {\small (1.69)} & {\small (1.68)} & {\small (1.53)}\\
 \multicolumn{6}{c}{}\\
\end{tabular}
\begin{tabular}{c|ccc|ccc}
 & \multicolumn{3}{c}{$b_3$} & \multicolumn{3}{c}{$b_4$}\\
 $n$ & CV & PCO & Or & CV & PCO & Or\\ \hline
 $250$ & 10.6 & 19.4 & 11.1 & 4.54 & 7.15 & 6.29\\
 & {\small (10.6)} & {\small (19.4)} & {\small (11.1)} & {\small (4.37)} & {\small (7.95)} & {\small (9.28)}\\
 $500$ & 5.71 & 5.84 & 5.51 & 2.52 & 2.84 & 2.97\\
 & {\small (3.45)} & {\small (5.26)} & {\small (7.10)} & {\small (2.16)} & {\small (2.07)} & {\small (2.89)}\\
 $1000$ & 3.17 & 2.70 & 2.47 & 1.50 & 1.38 & 1.41\\
 & {\small (2.02)} & {\small (1.69)} & {\small (1.61)} & {\small (1.44)} & {\small (1.14)} & {\small (1.17)}\\
 \multicolumn{6}{c}{}\\
\end{tabular}
\caption{100*MISE (with 100*std in parenthesis below) for the estimation of $b_i$, $i=1, \dots, 4$, 200 repetitions, $X\sim\gamma(3,2)/5$, $\sigma = 0.7$. CV and PCO are the two competing methods. Column "Or" gives the average of ISE for the ratio of the two best estimators of $bf$ and $f$ in the collection.}\label{tab_b_XGamma07}
\end{table}
\end{document}